\documentclass[10pt, a4paper, notitlepage]{article}
\pdfoutput=1

\usepackage{biblatex}
\usepackage{bbm}

\usepackage{etoolbox}
\usepackage{tikz}
\usetikzlibrary{calc}
\usetikzlibrary{cd}
\usetikzlibrary{decorations.markings}
\usetikzlibrary{decorations.pathreplacing}
\usetikzlibrary{decorations.pathmorphing}
\usetikzlibrary{decorations.text}
\usetikzlibrary{arrows.meta}
\usetikzlibrary{arrows}
\usetikzlibrary{positioning}

\usepackage{geometry} 
\usepackage{tocloft} 
\usepackage{fancyhdr} 
\usepackage{longtable}
\usepackage{subcaption}
\usepackage{enumitem}
\usepackage{amssymb}
\usepackage{amsmath}
\usepackage{stackrel} 
\usepackage{uniinput}
\usepackage{amsthm}
\usepackage{stmaryrd}
\usepackage{mathtools}
\usepackage{fouridx}
\usepackage{xparse}
\usepackage{bm}
\usepackage{aliascnt}
\usepackage{import}

\usepackage{hyperref}
\theoremstyle{definition}
\newtheorem{theorem}{Theorem}
\let\emph\textbf

\newaliascnt{remark}{theorem}
\newaliascnt{definition}{theorem}
\newaliascnt{proposition}{theorem}
\newaliascnt{lemma}{theorem}
\newaliascnt{corollary}{theorem}
\newaliascnt{example}{theorem}
\newaliascnt{convention}{theorem}
\newaliascnt{todo}{theorem}

\newtheorem{remark}[remark]{Remark}
\newtheorem{definition}[definition]{Definition}
\newtheorem{proposition}[proposition]{Proposition}
\newtheorem{lemma}[lemma]{Lemma}

\aliascntresetthe{remark}
\aliascntresetthe{definition}
\aliascntresetthe{proposition}
\aliascntresetthe{lemma}
\aliascntresetthe{corollary}
\aliascntresetthe{example}
\aliascntresetthe{convention}
\aliascntresetthe{todo}

\numberwithin{equation}{section}
\numberwithin{figure}{section}
\numberwithin{table}{section}
\makeatletter\let\c@table\c@figure\makeatother

\numberwithin{theorem}{section}
\numberwithin{remark}{section}
\numberwithin{definition}{section}
\numberwithin{proposition}{section}
\numberwithin{lemma}{section}
\numberwithin{corollary}{section}
\numberwithin{example}{section}
\numberwithin{convention}{section}
\numberwithin{todo}{section}

\newcommand{\id}{\mathrm{id}}

\newcommand{\Hom}{\operatorname{Hom}}
\newcommand{\Ker}{\operatorname{Ker}}

\newcommand{\cat}[1]{\mathcal{#1}}

\newcommand{\htensor}{\widehat{\otimes}}

\newcommand{\MC}{\operatorname{MC}}

\newcommand{\Gtl}{\operatorname{Gtl}}
\newcommand{\Jac}{\operatorname{Jac}}

\def\H{\operatorname{H}}

\newcommand{\HH}{\operatorname{HH}}
\newcommand{\HC}{\operatorname{HC}}

\newcommand{\vspan}{\operatorname{span}}

\def\Im{\operatorname{Im}}

\newcommand{\running}{~|~}

\newcommand{\cA}{\mathcal{A}}

\newcommand{\tensor}{¤}

\newcommand{\Rpart}{\mathsf{part}}
\newcommand{\Rlist}{\mathsf{list}}
\newcommand{\sporadicP}{\mathbb{S}_{\mathsf{poly}}}
\newcommand{\sporadicG}{\mathbb{S}_{\mathsf{generic}}}
\newcommand{\GtlT}{\operatorname{Gtl}_2}
\newcommand{\odd}{\mathrm{o}}
\newcommand{\even}{\mathrm{e}}

\title{$ L_∞ $-spectral sequences for Hochschild cohomology}
\author{Jasper van de Kreeke}
\date{}

\DeclareRobustCommand{\SkipTocEntry}[5]{}

\begin{document}

\maketitle

\begin{abstract}
Spectral sequences are a common tool to compute cohomology spaces, but higher structure is often lost on the way. In this article we exhibit a strategy to retain the higher structure on the cohomology, which works in case the chain complex with higher structure is presented as the twisting of a simpler chain complex by a Maurer-Cartan element. This works particularly well in case of the Hochschild complex of a minimal $ A_∞ $-category, where the Hochschild complex can be seen as the twisting of the Hochschild complex of the underlying associative algebra or category. As an application, we recover the existing formality result for Hochschild cohomology of wrapped Fukaya categories of punctured surfaces, and provide an $ L_∞ $-quasi-isomorphism between the Hochschild complex and its cohomology.
\end{abstract}

\tableofcontents

\section{Introduction}
Spectral sequences are a classical tool to determine cohomology of double complexes. Starting from a bigraded vector space $ V $ with two anti-commuting differentials $ d_0, d_1 $, one takes cohomology first with respect to $ d_0 $, then with respect to $ d_1 $, and potentially with respect to further differentials:
\begin{center}
\begin{tikzpicture}
\path (0, 0) node[align=center] {$ E^0 = (V, d_0 + d_1) $ \\ Double complex};
\path (2, 0) node {\Large $ \rightsquigarrow $};
\path (4, 0) node[align=center] {$ E^1 = (\H(V, d_0), \overline{d_1}) $ \\ First page complex};
\path (6.3, 0) node {\Large $ \rightsquigarrow $};
\path (9, 0) node [align=center] {$ E^2 = (\H(\H(V, d), \overline{d_1}), d_2) $ \\ Second page complex};
\path (12, 0) node {\Large $ \rightsquigarrow ~~ \cdots $};
\end{tikzpicture}
\end{center}
Modern cohomology theories produce chain complexes $ (V, d_0 + d_1) $ which carry additional higher structure such as an $ A_∞ $-structure or $ L_∞ $-structure. The standard tool to transfer the higher structure to cohomology is by successive homological transfer from one page to the next.

In the present article, we exhibit an alternative approach which applies if the higher algebra structure can be written as a twisting of a simpler structure by a Maurer-Cartan element $ W $. We focus on the case of DG Lie algebras of the form $ (V, d_0 + [W, -], [-, -]) $, in which we obtain a second page that is rather different from classical spectral sequences:
\begin{center}
\begin{tikzpicture}
\path (-1, 0) node[align=center] {$ L_W = (V, d_0 + [W, -], [-, -]) $ \\ Spectral DGLA};
\path (1.8, 0) node {\Large $ \rightsquigarrow $};
\path (4, 0) node[align=center] {$ E^1 = (\H L)_{π^{\MC} (W)} $ \\ First $ L_∞ $-page};
\path (6.2, 0) node {\Large $ \rightsquigarrow $};
\path (9, 0) node [align=center] {$ E^2 = \H((\H L)_{π^{\MC} (W)}) $ \\ Second $ L_∞ $-page};
\end{tikzpicture}
\end{center}
As an application, we compute the Hochschild cohomology of $ A_∞ $-gentle algebras and recover the formality result of \cite{Paper-I} in a more elegant fashion.

\addtocontents{toc}{\SkipTocEntry}
\subsection{Classical transfer of higher structure}
Assume one is given a chain complex with higher multiplicative structure, and a spectral sequences which computes the cohomology merely as a vector space. Then the task at hand is to compute the higher multiplicative structure on the cohomology such that there is a quasi-isomorphism between the complex and the cohomology that respects the higher structure.

The standard tool for this is to transfer the higher structure from the complex successively to the pages $ E^1 $, $ E^2 $ and eventually to the cohomology space. For instance, Cirici and Sopena-Gilboy \cite{Cirici-Sopena} consider the spectral sequence connecting Dolbeault cohomology and de Rham cohomology with the help of a filtered Kadeishvili theorem. Herscovich \cite{Herscovich} provides a further refined theory of $ A_∞ $-structures on spectral sequences. In all cases, the differential on the first page is given by the descent $ \overline{d_1} $ of the second differential to the first page.

\addtocontents{toc}{\SkipTocEntry}
\subsection{The case of twisted structures}
In some situations it seems more natural to consider a different differential on the first page. Indeed, consider the case of a DGLA $ L = (V, d_0, [-, -]) $ and its twisting $ L_W = (V, d_0 + [W, -], [-, -]) $ by a Maurer-Cartan element $ W ∈ \MC(L) $. Homological transfer provides a quasi-isomorphism of $ L_∞ $-algebras $ π: L → \H L $ and a map of Maurer-Cartan elements $ π^{\MC}: \MC(L) → \MC(\H L) $. Embracing $ π $ with $ π^{\MC} (W) $ we get a morphism of $ L_∞ $-algebras $ π_W: L_W → (\H L)_{π^{\MC} (W)} $.

In contrast to classical spectral sequences, the differential on the first page $ (\H L)_{π^{\MC}} $ is given by the embraced differential instead the descent of $ d_1 = [W, -] $ to cohomology. Moreover the map $ π_W $ need not be a quasi-isomorphism. But if it is, then the composition of $ π_W $ and the projection to the minimal model of $ (\H L)_{π^{\MC} (W)} $ yield a quasi-isomorphism between $ L_W $ and a minimal $ L_∞ $-algebra. In this sense, our type of non-classical spectral sequence always degenerates on the second page.

\addtocontents{toc}{\SkipTocEntry}
\subsection{Hochschild cohomology of $ A_∞ $-categories}
The particular application which we have in mind concerns the Hochschild cohomology of an arbitrary minimal $ A_∞ $-category. In this case, the $ A_∞ $-structure can be written as the sum $ μ = μ^2 + μ^{≥3} $ of the ordinary product and the higher products. Therefore the differential on the Hochschild complex splits into two parts as $ d = [μ^2, -] + [μ^{≥3}, -] $, which precisely presents the Hochschild complex of the $ A_∞ $-category as the twisting of the Hochschild complex of the associative category by the Maurer-Cartan element $ μ^{≥3} $. Thanks to numerous results and tools for Hochschild cohomology of associative algebras \cite{Barmeier-Wang, Kontsevich-formality, Bocklandt, Chouhy-Solotar, Bardzell}, it becomes possible to compute Hochschild cohomology and its $ L_∞ $-structure of $ A_∞ $-categories via a two-step procedure.

In this article, we demonstrate this strategy in the case of wrapped Fukaya categories of punctured surfaces, also known as gentle algebras $ \Gtl Σ $ \cite{Bocklandt}. The computation is based on the description of Hochschild cohomology of $ \Gtl Σ $ as associative category due to Bocklandt \cite{Bocklandt}. The approach via spectral sequences was already anticipated in \cite[Theorem 12.22]{Bocklandt-book}, but remained without proof. The Hochschild cohomology of $ \Gtl Σ $ as $ A_∞ $-category was finally computed in \cite{Paper-I}, but the $ L_∞ $-structure was proven to be formal only through a fragile grading argument. The present article aims to recover this formality result in a more elegant way.

\addtocontents{toc}{\SkipTocEntry}
\subsection{Further applications}
It seems likely that the approach exhibited here applies to further DGLAs and situations outside of the DGLA context as well. For instance, consider the Landau-Ginzburg model $ (\Jac Q, ℓ) $ where $ Q $ is a dimer model and $ ℓ ∈ Z(\Jac Q) $ is the standard central element. Wong \cite{Wong} computed the compactly supported Hochschild cohomology $ \HH(\Jac Q, ℓ) $ by means of a spectral sequence in which the second differential is given by commuting with $ ℓ $ in $ \HH(\Jac Q) $. The higher structure on cohomology is missing in \cite{Wong} and may be recovered by interpreting $ \HC(\Jac Q, ℓ) $ as a spectral DGLA. Further situations outside of the DGLA context seem possible, perhaps existing transfers of $ A_∞ $-structure can be reformulated through the lens of twisted structures presented in the present article.

\subsection*{Acknowledgements}
The ideas for the present paper were born during the author's PhD project at the University of Amsterdam under supervision of Raf Bocklandt. Writing the paper was supported by NWO Rubicon grant 019.232EN.029.

\section{Preliminaries}
\label{sec:prelim}
In this section, we recall preliminaries on DG Lie algebras, $ A_∞ $-categories and gentle algebras. We also fix notation and sign conventions. We work over a field of characteristic zero throughout and write $ ℂ $. Textbook references for these materials include \cite{Bocklandt-book}.

\addtocontents{toc}{\SkipTocEntry}
\subsection{The Hochschild complex}
In this section, we recall material on the Hochschild complex of $ A_∞ $-categories. We start by recalling $ A_∞ $-categories, $ L_∞ $-algebras, DGLAs and the Hochschild complex in particular. The sign conventions for $ A_∞ $-categories follow the standard symplectic convention chosen for instance in \cite{HKK}. The sign conventions for $ L_∞ $-algebras are such that an $ L_∞ $-algebra with vanishing higher brackets is precisely a DGLA, as in \cite[version 3]{Barmeier-Wang}. While we present the correct signs for $ A_∞ $-categories, the accuracy of the signs for $ L_∞ $-algebras is not essential for the present paper and will remain sloppy.

\begin{definition}
A ($ ℤ $- or $ ℤ/2ℤ $-graded, strictly unital) \emph{$ A_∞ $-category} $ \cat C $ consists of a collection of objects together with $ ℤ $- or $ ℤ/2ℤ $-graded hom spaces $ \Hom(X, Y) $, distinguished identity morphisms $ \id_X ∈ \Hom^0(X, X) $ for all $ X ∈ \cat C $, together with multilinear higher products
\begin{equation*}
μ^k: \Hom(X_k, X_{k+1}) ¤ … ¤ \Hom(X_1, X_2) → \Hom(X_1, X_{k+1}), \quad k ≥ 1
\end{equation*}
of degree $ 2-k $ such that the $ A_∞ $-relations and strict unitality axioms hold: For every compatible morphisms $ a_1, …, a_k $ we have
\begin{align*}
& \sum_{0 ≤ j < i ≤ k} (-1)^{‖a_n‖ + … + ‖a_1‖} μ(a_k, …, μ(a_i, …, a_{j+1}), a_j, …, a_1) = 0, \\
& μ^2 (a, \id_X) = a, ~ μ^2 (\id_Y, a) = (-1)^{|a|} a, ~ μ^{≥3} (…, \id_X, …) = 0.
\end{align*}
\end{definition}

Next let us recall the notion of DG Lie algebras and Maurer-Cartan elements.

\begin{definition}
A \emph{DG Lie algebra} (DGLA) is a $ ℤ $- or $ ℤ/2ℤ $-graded vector space $ L $ together with a differential $ d: L^i → L^{i+1} $ and a bracket $ [-, -]: L × L → L $ of degree zero satisfying skew-symmetry, Leibniz rule and Jacobi identity:
\begin{align*}
& [a, b] = (-1)^{|a||b| + 1} [b, a], \\
& d([a, b]) = [da, b] + (-1)^{|a|} [a, db], \\
& (-1)^{|a||c|} [a, [b, c]] + (-1)^{|b||a|} [b, [c, a]] + (-1)^{|c||b|} [c, [a, b]] = 0.
\end{align*}
A \emph{Maurer-Cartan element} of $ L $ is an element $ ν ∈ L^1 $ which satisfies the Maurer-Cartan equation
\begin{equation*}
dν + \frac{1}{2} [ν, ν] = 0.
\end{equation*}
The set of Maurer-Cartan elements of $ L $ is denoted $ \MC(L) $.
\end{definition}

\begin{remark}
The notion of DG Lie algebras exists in both $ ℤ $-graded and $ ℤ/2ℤ $-graded context. In the $ ℤ/2ℤ $-graded case, we write the $ L = L^{\even} ⊕ L^{\odd} $ but may equally use integers to indicate the parity. For instance, in case the DGLA $ L $ is $ ℤ/2ℤ $-graded, a Maurer-Cartan element is supposed to lie in $ L^{\odd} $.
\end{remark}

\begin{definition}
Let $ \cat C $ be a $ ℤ $- or $ ℤ/2ℤ $-graded $ A_∞ $-category. Then its Hochschild complex $ \HC(\cat C) $ is the graded vector space
\begin{equation*}
\HC(\cat C) = \prod_{\substack{X_1, …, X_{k+1} ∈ \cat C \\ k ≥ 0}} \Hom\big(\Hom_{\cat C} (X_k, X_{k+1})[1] \tensor … \tensor \Hom_{\cat C} (X_1, X_2)[1], ~ \Hom_{\cat C} (X_1, X_{k+1})[1]\big).
\end{equation*}
Here $ [1] $ denotes the left-shift and $ ‖a‖ = |a| - 1 $ denotes the reduced degree of a morphism $ a ∈ \Hom_{\cat C} (X, Y) $. The grading $ ‖·‖ $ on $ \HC(\cat C) $ is the one induced from the shifted degrees of the hom spaces of $ \cat C $. In other words, we have
\begin{equation*}
‖η(a_k, …, a_1)‖ = ‖η‖ + ‖a_k‖ + … + ‖a_1‖, \quad η ∈ \HC(\cat C).
\end{equation*}
For $ η, ω ∈ \HC(\cat C) $, the Gerstenhaber product $ μ · ω ∈ \HC(\cat C) $ is defined as
\begin{equation*}
(η · ω) (a_k, …, a_1) = \sum (-1)^{(∥a_l∥ + … + \Vert a_1 \Vert)∥ω∥} η(a_k, …, ω(…), a_l, …, a_1).
\end{equation*}
The \emph{Hochschild DGLA} is the following $ ℤ $- or $ ℤ/2ℤ $-graded DGLA structure on $ \HC(\cat C) $: The bracket on $ \HC(\cat C) $ is the Gerstenhaber bracket
\begin{equation*}
[η, ω] = η · ω - (-1)^{‖ω‖ ‖η‖} ω · η.
\end{equation*}
The differential on $ \HC(\cat C) $ consists of commuting with the product $ μ_{\cat C} ∈ \HC^1 (\cat C) $:
\begin{equation*}
dν = [μ_{\cat C}, ν].
\end{equation*}
\end{definition}

We will often encounter a situation in which we provide a list $ R_{\Rlist} $ of individual Hochschild cochains and require to construct a space which contains all of them. In these situations the vector span is often not sufficient and we rather need its completion. We fix this terminology as follows.

\begin{definition}
\label{def:prelim-completion}
Let $ \cat C $ be an $ A_∞ $-category. Let $ R_{\Rlist} ⊂ \HC(\cat C) $ be a set of Hochschild cochains which are individually concentrated in some specific length. For every $ m ≥ 0 $ denote by $ R_{\Rlist, m} ⊂ R_{\Rlist} $ the subset of cochains of length $ m $. Then the \emph{completion} space $ \overline{R_{\Rlist}} ⊂ \HC(\cat C) $ is defined as
\begin{equation*}
\overline{R_{\Rlist}} = \prod_{m = 0}^∞ \vspan(R_{\Rlist, m}) ⊂ \HC(\cat C).
\end{equation*}
\end{definition}

Next we recall $ L_∞ $-algebras and their morphisms. With the present sign convention, a DGLA is simply an $ L_∞ $-algebra without higher products. In particular, the Hochschild DGLA can automatically be regarded as an $ L_∞ $-algebra.

\begin{definition}
An \emph{$ L_∞ $-algebra} is a graded vector space $ L $ together with multilinear maps
\begin{equation*}
l_k: \underbrace{L × … × L}_{k \text{ times}} → L
\end{equation*}
of degree $ 2-k $ which are graded skew-symmetric and satisfy the higher Jacobi identities:
\begin{align*}
&l_k (x_{s(1)}, …, x_{s(k)}) = χ(s) l_k (x_1, …, x_k), \\
&\sum_{\substack{i+j = k+1 \\ i, j ≥ 1}} \sum_{s ∈ S_{i, k-i}} (-1)^{i(n-i)} χ(s) l_j (l_i (x_{s(1)}, …, x_{s(i)}), x_{s(i+1)}, …, x_{s(k)}) = 0.
\end{align*}
Here $ S_{i, k-i} $ denotes the set of shuffles, i.e.~$ s ∈ S_k $ with $ s(1) < … < s(i) $ and $ s(i+1) < … < s(k) $. The sign $ χ(s) $ is the product of the signum of $ s $ and the Koszul sign of $ s $. The Koszul sign of a transposition $ (i ~ j) $ is $ (-1)^{|x_i| |x_j|} $, and this rule is extended multiplicatively to arbitrary permutations.
\end{definition}

Morphisms between $ L_∞ $-algebras are more flexible than between DGLAs, since $ L_∞ $-morphisms are allowed to have higher components, similar to $ A_∞ $-functors.

\begin{definition}
A \emph{morphism of $ L_∞ $-algebras} $ φ: L → L' $ is given by multilinear maps
\begin{equation*}
φ^k: L × … × L → L
\end{equation*}
of degree $ 1-k $ for all $ k ≥ 1 $ such that $ φ(x_{s(1)}, …, x_{s(k)}) = χ(s) φ(x_1, …, x_k) $ for any $ s ∈ S_k $ and
\begin{align*}
&\sum_{\substack{i+j = k+1 \\ i, j ≥ 1}} \sum_{s ∈ S_{i, n-i}} (-1)^{i(k-i)} χ(s) φ(l(x_{s(1)}, …, x_{s(i)}), x_{s(i+1)}, …, x_{s(k)}) \\
&= \sum_{\substack{1 ≤ r ≤ k \\ i_1 + … + i_r = k}} \sum_t (-1)^u χ(t) l'_r (φ(x_{t(1)}, …, x_{t(i_1)}), …, φ(x_{t(i_1 + … + i_{r-1} + 1)}, …, x_{t(k)})),
\end{align*}
where $ t $ runs over all $ (i_1, …, i_r) $-shuffles for which
\begin{equation*}
t(i_1 + … + i_{l-1} + 1) < t(i_1 + … + i_l + 1).
\end{equation*}
and $ u = (r-1) (i_1 - 1) + … + 2 (i_{r-2} - 1) + (i_{r-1} - 1) $. A morphism $ φ: L → L' $ of $ L_∞ $-algebras is a \emph{quasi-isomorphism} if $ φ^1 $ is a quasi-isomorphism of complexes.
\end{definition}

Every $ L_∞ $-algebra $ L $ possesses a minimal model $ \H L $ whose differential vanishes. The minimal model is uniquely determined up to $ L_∞ $-isomorphism and can be determined by means of an explicit construction. References for the analogous statement for $ A_∞ $-algebras include \cite[Section 3.2]{Bocklandt-book} and \cite{Paper-IIA}. We shall fix the terminology as follows.

\begin{definition}
A \emph{homological splitting} of an $ L_∞ $-algebra $ L $ consists of direct sum decomposition $ H ⊕ \Im(d) ⊕ R $ such that $ H ⊕ \Im(d) = \Ker(d) $. The space $ H $ is the \emph{cohomology space} and the space $ R $ is the \emph{complement space}. The inclusion and projection maps are denoted by $ i: H → L $ and $ π: L → H $. The \emph{codifferential} $ h: L → R $ is given by $ h(d(ν)) = ν $ for $ ν ∈ R $ and $ h(ν) = 0 $ for $ ν ∈ H ⊕ R $.
\end{definition}

We recall now the construction of minimal models of $ L_∞ $-algebras. This construction is also known as homological perturbation, and the particular decorated trees are known as Kadeishvili trees. We denote by $ \mathcal{T}_k $ the set of rooted planar trees with $ k $ leaves in which every node has at least $ 2 $ children. For every $ T ∈ \mathcal{T}_k $ and elements $ a_1, … a_k ∈ H $ we define the evaluation $ T_π (a_1, …, a_k) $ as follows: The leaves are decorated with $ a_1, …, a_k $. Every non-leaf non-node is decorated with $ h l^m $ where $ m $ is the number of children at that node. The root is decorated with $ π l^m $ where $ m $ is its number of children. The evaluation $ T_h (a_1, …, a_k) $ is defined similarly by decorating the root with $ h l^m $ instead of $ π l^m $.

\begin{theorem}
Let $ L $ be an $ L_∞ $-algebra with homological splitting $ L = H ⊕ \Im(d) ⊕ R $. Then a minimal $ L_∞ $-structure on $ \H L ≔ H $ is constructed by means of the following formula:
\begin{equation*}
l^k_{\H L} (a_1, …, a_k) = \sum_{T ∈ \mathcal{T}_k} \sum_{σ ∈ S_k} c_{T,σ} T_π (a_{σ(1)}, …, a_{σ(k)}).
\end{equation*}
A quasi-isomorphism $ i: \H L → L $ is constructed by means of
\begin{equation*}
i^k (a_1, …, a_k) = \sum_{T ∈ \mathcal{T}_k} d_{T,σ} T_h (a_1, …, a_k).
\end{equation*}
\end{theorem}

\begin{remark}
The scalars $ c_{T,σ} $ and $ d_{T, σ} $ can be determined explicitly. We have $ μ^1_{\H L} = 0 $ and $ l^2_{\H L} (a_1, a_2) = π l^2 (a_1, a_2) $. Moreover $ i^1 (a_1) = a_1 $.
\end{remark}

\addtocontents{toc}{\SkipTocEntry}
\subsection{Gentle algebras}
In this section, we recall the construction of the $ A_∞ $-gentle algebra from \cite{Bocklandt, Paper-I}. The first step is to define an associative category $ \GtlT Σ $. The second step is to construct an $ A_∞ $-structure on this category, and we denote the resulting $ A_∞ $-category by $ \Gtl Σ $. By means of abuse of terminology, we shall refer to these categories as gentle algebras, for historical reasons \cite{Bocklandt}.
\begin{center}
\begin{tikzpicture}
\path (0, 0) node[align=center] {$ (Σ, M) $ \\ Punctured surface};
\path (2.3, 0) node {\Large $ \rightsquigarrow $};
\path (5, 0) node[align=center] {$ \GtlT Σ $ \\ Classical gentle algebra};
\path (7.5, 0) node {\Large $ \rightsquigarrow $};
\path (9.5, 0) node [align=center] {$ \Gtl Σ $ \\ $ A_∞ $-gentle algebra};
\end{tikzpicture}
\end{center}

\begin{definition}
A \emph{punctured surface} is a closed oriented surface $ Σ $ with a finite set of punctures $ M ⊂ Σ $. We assume that $ |M| ≥ 1 $, or $ |M| ≥ 3 $ if $ Σ $ is a sphere.
\end{definition}

The assumptions on $ |M| $ are cosmetical and explained in \cite{Paper-IIB}.

\begin{definition}
Let $ (Σ, M) $ be a punctured surface. An \emph{arc} in $ Σ $ is a not necessarily closed curve $ γ: [0, 1] → Σ $ running from one puncture to another. An \emph{arc system} on a punctured surface is a finite collection of arcs such that the arcs meet only at the set $ M $ of punctures. Intersections and self-intersections are not allowed. The arc system satisfies the \emph{no monogons or digons} condition \emph{[NMD]} if
\begin{itemize}
\item No arc is a contractible loop in $ Σ \setminus M $.
\item No pair of distinct arcs is homotopic in $ Σ \setminus M $.
\end{itemize}
An arc system is \emph{full} if the arcs cut the surface into contractible pieces, which we call \emph{polygons}. The set of polygons is denoted $ Σ_2 $.
\end{definition}

We are now ready to recall the construction of the category $ \GtlT Σ $. Let $ (Σ, M) $ be a punctured surface and $ \cA $ be a full arc system for $ Σ $ which satisfies [NMD]. The \emph{classical gentle algebra} $ \GtlT Σ $ is a graded associative category defined as follows. Its objects are the arcs $ a ∈ \cA $. A basis for the hom space $ \Hom_{\GtlT Σ} (a, b) $ is given by the set of all angles around punctures from $ a $ to $ b $. This includes \emph{empty angles}, which are the \emph{identities} on the arcs. An angle is \emph{indecomposable} if it is the interior angle of a polygon, otherwise it is \emph{decomposable}. The hom spaces of $ \GtlT Σ $ are not finite-dimensional, in contrast to what is classically called a gentle algebra. The $ ℤ/2ℤ $-grading on $ \GtlT Σ $ is given by the declaring an interior angle of a polygon to have odd degree if the orientations of the two incident arcs agree relative to the boundary of the polygon, see \autoref{fig:prelim-angledeg}. The product $ μ^2 (α, β) $ of two angles $ α, β $ is defined as the concatenation $ αβ $ of $ α $ and $ β $ if both angles wind around the same puncture and $ α $ starts where $ β $ ends:
\begin{equation*}
μ^1 ≔ 0, \quad μ^2 (α, β) ≔ (-1)^{|β|} αβ.
\end{equation*}

\begin{figure}
\centering
\begin{subfigure}{0.2\linewidth}
\centering
\begin{tikzpicture}
\path[draw, ->] (-1, 1) -- (0, 0) coordinate[midway] (2);
\path[draw, ->] (0, 0) -- (1, 1) coordinate[midway] (1);
\path[draw, ->, bend right] (1) to node[midway, below] {$ α $} (2);
\end{tikzpicture}
\caption{odd}
\end{subfigure}
\begin{subfigure}{0.2\linewidth}
\centering
\begin{tikzpicture}
\path[draw, <-] (-1, 1) -- (0, 0) coordinate[midway] (2);
\path[draw, <-] (0, 0) -- (1, 1) coordinate[midway] (1);
\path[draw, ->, bend right] (1) to node[midway, below] {$ α $} (2);
\end{tikzpicture}
\caption{odd}
\end{subfigure}
\begin{subfigure}{0.2\linewidth}
\centering
\begin{tikzpicture}
\path[draw, <-] (-1, 1) -- (0, 0) coordinate[midway] (2);
\path[draw, ->] (0, 0) -- (1, 1) coordinate[midway] (1);
\path[draw, ->, bend right] (1) to node[midway, below] {$ α $} (2);
\end{tikzpicture}
\caption{even}
\end{subfigure}
\begin{subfigure}{0.2\linewidth}
\centering
\begin{tikzpicture}
\path[draw, ->] (-1, 1) -- (0, 0) coordinate[midway] (2);
\path[draw, <-] (0, 0) -- (1, 1) coordinate[midway] (1);
\path[draw, ->, bend right] (1) to node[midway, below] {$ α $} (2);
\end{tikzpicture}
\caption{even}
\end{subfigure}
\caption{Degree of an angle $ α $ in $ \GtlT Σ $}
\label{fig:prelim-angledeg}
\end{figure}

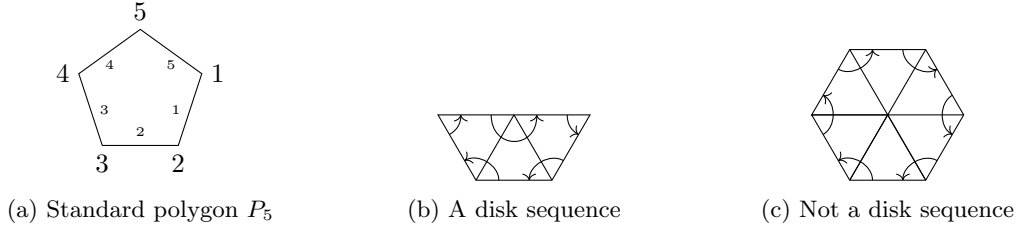
\begin{figure}
\centering
\begin{subfigure}[b]{0.3\linewidth}
\centering
\begin{tikzpicture}
\path[draw] (0, 0) node[below] {3} -- ++(right:1) node[midway, above] {\tiny 2} node[below] {2} -- ++(72:1) node[midway, left] {\tiny 1} node[right] {1} -- ++(144:1) node[midway, below] {\tiny 5} node[above] {5} -- ++(216:1) node[midway, below] {\tiny 4} node[left] {4} -- ++(-72:1) node[midway, right] {\tiny 3};
\end{tikzpicture}
\caption{Standard polygon $ P_5 $}
\label{fig:prelim-gtl-infty-P5}
\end{subfigure}
\begin{subfigure}[b]{0.3\linewidth}
\centering
\begin{tikzpicture}
\path[draw] (0, 0) -- ++(left:1) coordinate[pos=0.3] (1-start) coordinate[pos=0.7] (2-end) -- ++(300:1) coordinate[pos=0.3] (2-start) coordinate[pos=0.7] (3-end) -- ++(right:1) coordinate[pos=0.3] (3-start) coordinate[pos=0.7] (4-end) -- ++(60:1) coordinate[pos=0.3] (4-start) coordinate[pos=0.7] (5-end) -- ++(left:1) coordinate[pos=0.3] (5-start) coordinate[pos=0.7] (1-end);
\path[draw] (0, 0) -- ++(240:1) (0, 0) -- ++(300:1);
\foreach \i in {2, 5} \path[draw, ->, bend right] (\i-start) to (\i-end);
\foreach \i in {3, 4} \path[draw, ->, bend right=60] (\i-start) to (\i-end);
\path[draw, ->, bend right=90, looseness=2] (1-start) to (1-end);
\end{tikzpicture}
\caption{A disk sequence}
\label{fig:prelim-gtl-infty-diskexample}
\end{subfigure}
\begin{subfigure}[b]{0.3\linewidth}
\centering
\begin{tikzpicture}
\path[draw] (0, 0) -- ++(left:1) coordinate (A) -- ++(60:1) coordinate (B) coordinate[pos=0.3] (3-end) coordinate[pos=0.7] (4-start) -- ++(right:1) coordinate[pos=0.3] (4-end) coordinate[pos=0.7] (5-start) coordinate (C) -- ++(300:1) coordinate[pos=0.3] (5-end) coordinate[pos=0.7] (6-start) coordinate (D) -- ++(240:1) coordinate[pos=0.3] (6-end) coordinate[pos=0.7] (7-start) coordinate (E) -- ++(left:1) coordinate[pos=0.4] (1-end) coordinate[pos=0.3] (7-end) coordinate[pos=0.7] (2-start) coordinate[pos=0.6] (8-start) coordinate (F) -- ++(120:1) coordinate[pos=0.3] (2-end) coordinate[pos=0.7] (3-start);
\path[draw] (0, 0) -- (F) coordinate[pos=0.6] (8-end) coordinate[pos=0.3] (9-start);
\path[draw] (0, 0) -- (E) coordinate[pos=0.6] (1-start) coordinate[pos=0.3] (9-end);
\foreach \i in {2, 3, 4, 5, 6, 7} \path[draw, ->, bend right=60] (\i-start) to (\i-end);
\path[draw] (A) -- (D) (B) -- (E) (C) -- (F);
\end{tikzpicture}
\caption{Not a disk sequence}
\label{fig:prelim-gtl-infty-no-disk}
\end{subfigure}
\caption{Illustration of immersed disks}
\label{fig:prelim-gtl-infty-diskillustration}
\end{figure}

The next step is to define an $ A_∞ $-structure on $ \GtlT Σ $ which turns the category into the $ A_∞ $-gentle algebra denoted by $ \Gtl Σ $. The higher products $ μ^{≥3} $ on $ \Gtl Σ $ are defined in terms of \emph{immersed disks}. Roughly speaking, a discrete immersed disk may either be a polygon, or a sequence of polygons stitched together. We make this precise by regarding immersions of the standard polygon $ P_k $, depicted in \autoref{fig:prelim-gtl-infty-diskillustration}:

\begin{definition}
Let $ Σ $ be dimer. An \emph{immersed disk} in $ Σ $ consists of an oriented immersion $ D: P_k → |Σ| $ of a standard polygon $ P_k $ into the surface, such that
\begin{itemize}
\item The edges of the polygon are mapped to a sequence of arcs.
\item The immersion does not cover any punctures.
\end{itemize}
The immersion mapping itself is only taken up to reparametrization. The sequence of \emph{interior angles} of $ D $ is the sequence of angles in $ Σ $ given as images of the interior angles of $ P_k $ under the map $ D $. An angle sequence $ α_1, …, α_k $ is a \emph{disk sequence} if it is the sequence of interior angles of some immersed disk.
\end{definition}

We can now define the category $ \Gtl Σ $ as follows:

\begin{definition}
Let $ (Σ, M) $ be a punctured surface and $ \cA $ be a full arc system with [NMD]. Then the $ A_∞ $-\emph{gentle algebra} $ \Gtl Σ $ of $ Σ $ is the $ A_∞ $-category with objects the arcs $ a ∈ \cA $, hom spaces spanned by angles, and $ A_∞ $-product $ μ $ defined by $ μ^1 = 0 $ and $ μ^2 (α, β) = (-1)^{|β|} αβ $. To define $ μ^{k≥3} $, let $ α_1, …, α_k $ be any disk sequence, let $ β $ be any angle such that $ β α_1 ≠ 0 $, and let $ γ $ be any angle such that $ α_k γ ≠ 0 $. Then
\begin{equation*}
μ^k (β α_k, …, α_1) ≔ β, \quad μ^k (α_k, …, α_1 γ) ≔ (-1)^{|γ|} γ.
\end{equation*}
The higher products vanish on all angle sequences other than these.
\end{definition}

\addtocontents{toc}{\SkipTocEntry}
\subsection{Hochschild cohomology of $ \GtlT Σ $}
Let us recall results on the Hochschild cohomology of $ \GtlT Σ $ from \cite{Bocklandt-book} and \cite[Theorem 12.19]{Bocklandt-book}. The idea is that the category $ \GtlT Σ $ is effectively a graded quiver algebra with relations, and its Hochschild cohomology can be studied with classical tools \cite{Bocklandt}.

We start by recalling the two gradings on $ \HH(\GtlT Σ) $. The first grading is the classical grading by length of cochains. The second grading is the $ ℤ/2ℤ $-grading into even and odd cochains, coming from the reduced degree $ ‖·‖ $ on the Hochschild complex:
\begin{align*}
\HH(\GtlT Σ) &= \HH^0 (\GtlT Σ) ⊕ \HH^1 (\GtlT Σ) ⊕ …, \\
\HH(\GtlT Σ) &= \HH^{\even} (\GtlT Σ) ⊕ \HH^{\odd} (\GtlT Σ).
\end{align*}
We now fixing several pieces of notation and terminology. The \emph{full turn} $ ℓ_q $ around a puncture $ q ∈ M $ is the sum $ ℓ_q = \sum_a ℓ_{q, a} $ over all full turns $ ℓ_{q, a} $ winding around $ q $ and starting at all half-arcs incident at $ a $. Some specific cochains have only a 0-adic or a 1-adic component. Those with only a 0-adic are necessarily central elements, and those with 1-adic component are derivations in the sense that $ ν (αβ) = α ν (β) + α ν (β) $.

The set $ \sporadicG $ is the \emph{generic sporadic set} and is a finite set of scalar-valued functions $ λ $ on the set of angles in $ Σ $. Every function $ λ ∈ \sporadicG $ is additive in the sense that $ λ(αβ) = λ(α) + λ(β) $ whenever $ αβ ≠ 0 $, and satisfies the condition $ \sum_{i = 1}^l λ(α_i) = 0 $ whenever $ α_1, …, α_l $ is an elementary polygon. The set $ \sporadicG $ is chosen precisely so that the cocycles $ ν_λ $, $ ν_B $ and $ ν_{q, k, \odd} $ together form a basis for the 0-adic part of Hochschild cohomology.

There is another finite set $ \sporadicP = \{λ_B\}_{B ∈ Σ_2} $ of additive scalar-valued functions $ λ_B $. For every elementary polygon $ B ∈ Σ_2 $ with interior angles $ β_1, …, β_l $, the function $ λ_B $ has the property that $ \sum_{i = 1}^l λ(β_i) = 1 $, and $ λ $ vanishes on all elementary angles not contained in $ B $.

Finally, for every puncture $ q ∈ M $, one fixes an arbitrary indecomposable angle $ α_q $ which winds around $ q $. We are now ready to recall the following set of Hochschild cocycles. The cocycles of length $ 2 $ and higher were determined in \cite{Bocklandt}, and cocycles in length $ 0, 1 $ were provided in \cite[Theorem 12.19]{Bocklandt-book}.

\begin{definition}
The cochain $ ν_1 $ is simply given by identity element
\begin{equation*}
ν_1^0 = \sum_{a ∈ \cA} \id_a.
\end{equation*}
In these terms, the cochains $ ν_{q, k, \odd} $ are given by
\begin{equation*}
ν_{q, k, \odd}^0 = ℓ_q^k.
\end{equation*}
The generic sporadic set $ \sporadicG $ is of cardinality $ |M|+1 $. For $ λ ∈ \sporadicG $, the cochain $ ν_λ $ only has a 1-adic component given by
\begin{equation*}
ν_λ (α) = λ(α) α.
\end{equation*}
The \emph{polygon sporadic cochains} $ ν_B $ only have a 1-adic component given by
\begin{align*}
ν_B (β) &= λ_B (β) β.
\end{align*}
The cochains $ ν_{q, k, \even} $ are given as follows. Let $ α_1, …, α_l $ be the angles around $ q $ in a chosen order.
\begin{align*}
ν_{q, k, \even} (α_q) &= ℓ_q^k α_q, \\
ν_{q, k, \even} (α) &= 0 \text{ for all indecomposable angles} α ≠ α_q.
\end{align*}
Let $ B $ be an elementary polygon given by angles $ β_1, …, β_l $ where $ β_i: b_i → b_{i+1} $. Let $ k ≥ 1 $ be a natural number. The odd cochain $ ν_{B,k,o} ∈ \HC(\GtlT Σ) $ is defined as follows. If $ β $ is an angle such that $ β β_{i-1} ≠ 0 $ and $ γ $ is an angle such that $ β_i γ ≠ 0 $, then
\begin{align*}
ν_{B,k,o} (β β_{i+kl-1}, …, β_i) &= β, \\
ν_{B,k,o} (β_{i+kl-1}, …, β_i γ) &= (-1)^{|γ|} γ.
\end{align*}
The cochain $ ν_{B,k,o} $ vanishes on all other angle sequences. The even cochain $ ν_{B,k,e} $ is defined as follows. If $ β $ and $ γ $ are angles such that $ β β_i γ ≠ 0 $, then
\begin{align*}
ν_{B,k,e} (β β_{i+kl}, β_{i+kl-1}, …, β_i γ) &= β β_i γ.
\end{align*}
The cochain $ ν_{B,k,e} $ vanishes on all other angle sequences.
\end{definition}

Since the Hochschild cocycles $ ν_{B, k, \odd} $ and $ ν_{B, k, \even} $ were not explicitly constructed in \cite{Bocklandt, Bocklandt-book}, we shall check that they indeed satisfy the cocycle condition and illustrate their functionality.

\begin{lemma}
The elements $ ν_{B, k, \odd} $ and $ ν_{B, k, \even} $ are Hochschild cocycles.
\end{lemma}

\begin{proof}
Regard the cochain $ ν_{B,k,o} $ and let $ β', β $ be angles such that $ β' β β_{i-1} ≠ 0 $. Then
\begin{align*}
& (dν_{B,k,o})(β', β β_{i+kl-1}, …, β_i) \\
&= μ^2 (β', ν_{B,k,o} (β β_{i+kl}, …, β_i γ)) - (-1)^{‖β_i‖ + … + ‖β_{i+kl-2}‖} ν_{B,k,o} (μ^2 (β', β β_{i+kl-1}), …, β_i) \\
&= (-1)^{|β|} β' β + (-1)^{|β β_{i-1}| + ‖β_{i-1}‖} β' β = 0.
\end{align*}
We have used that $ ‖β_i‖ + … + ‖β_{i+kl-2}‖ \equiv ‖β_{i-1}‖ $. Let $ γ, γ' $ be angles such that $ β_i γ γ' ≠ 0 $. Then
\begin{align*}
&(dν_{B,k,o})(β_{i+kl-1}, …, β_i γ, γ') \\
&= (-1)^{‖γ'‖} μ^2 (ν_{B,k,o} (β_{i+kl-1}, …, β_i γ), γ') + ν_{B,k,o} (β_{i+kl-1}, …, μ^2 (β_i γ, γ')) \\
&= (-1)^{‖γ'‖ + |γ| + |γ'|} γ γ' + (-1)^{|γ'| + |γ γ'|} γ γ' = 0.
\end{align*}
It is obvious that $ dν_{B,k,o} $ vanishes on all other angle sequences.

We shall evaluate $ dν_{B,k,e} $ on the following types of angle sequences. Let $ β, γ $ be angles and $ γ' $ a non-empty angle such that $ β β_i γ γ' ≠ 0 $. Then
\begin{align*}
(dν_{B,k,e})(β β_{i+kl}, …, β_i γ, γ') &= μ^2 (ν_{B,k,e} (β β_{i+kl}, …, β_i γ), γ') - ν_{B,k,e} (β β_{i+kl}, …, μ^2 (β_i γ, γ')) \\
&= (-1)^{|γ'|} β β_i γ γ' - (-1)^{|γ'|} β β_i γ γ' = 0.
\end{align*}
Let $ β, γ $ be angles and $ γ' $ a non-empty angle such that $ β β_i γ γ' ≠ 0 $. Then
\begin{align*}
(dν_{B,k,e})(β', β β_{i+kl}, …, β_i γ) &= μ^2 (β', ν_{B,k,e} (β β_{i+kl}, …, β_i γ) \\
& \quad - (-1)^{‖β_i γ‖ + … + ‖β_{i+kl-1}‖} ν_{B,k,e} (μ^2 (β', β β_{i+kl}), …, β_i γ) \\
&= (-1)^{|β β_i γ|} β' β β_i γ - (-1)^{|γ| + |β β_i|} β' β β_i γ = 0.
\end{align*}
We have used that $ ‖β_i γ‖ + … + ‖β_{i+kl-1}‖ = |γ| $. It is obvious that $ dν_{B,k,e} $ vanishes on all other types of sequences.
\end{proof}

For the purposes of this paper, we require a description of the Hochschild cohomology $ \HH(\GtlT Σ) $ in the $ ℤ/2ℤ $-grading coming from the Hochschild complex, as opposed to the classical grading by length of chains. To illustrate the difference, note that adding up infinitely many Hochschild cochains of increasingly high length does not yield a Hochschild cohomology element in the classical grading, but does so in the $ ℤ/2ℤ $-grading. As explained in \cite[Section 12.3]{Bocklandt-book}, we arrive at the following description:

\begin{theorem}[\cite{Bocklandt-book}]
The Hochschild cohomology of $ \GtlT Σ $ in $ ℤ/2ℤ $-grading is given by the following completions:
\begin{align*}
\HH^{\even} (\GtlT Σ) &= \overline{\{ν_λ\}_{λ ∈ \sporadicG}, \{ν_B\}_{B ∈ Σ_2}, \{ν_{q, k, \even}\}_{q ∈ M, k ≥ 1}, \{ν_{B, k, \even}\}_{B ∈ Σ_2, k ≥ 1}}, \\
\HH^{\odd} (\GtlT Σ) &= \overline{ν_1, \{ν_{q, k, \odd}\}_{q ∈ M, k ≥ 1}, \{ν_{B, k, \odd}\}_{B ∈ Σ_2}}.
\end{align*}
\end{theorem}

\begin{remark}
To memorize this structure better, one may use the mnemonic
\begin{align*}
\HH^{\even} (\GtlT Σ) &≅ ℂ^{|M|+2g-1} ⊕ ℂ^{|Σ_2|} ⊕ \bigoplus_{q ∈ M} ℂ[ℓ_m] ℓ_m ∂_{ℓ_m} ⊕ \bigoplus_{B ∈ Σ_2} ℂ⟦ℓ_B⟧ ℓ_B ∂_{ℓ_B}, \\
\HH^{\odd} (\GtlT Σ) &≅ \frac{ℂ[ℓ_q \running q ∈ M]⟦ℓ_B \running B ∈ Σ_2⟧}{ℓ_i ℓ_j \running i ≠ j, ~ i,j ∈ M ∪ Σ_2}.
\end{align*}
\end{remark}

\section{Spectral sequences of DGLAs}
In this section, we exhibit our approach to minimal models of twisted DGLAs. The starting point is a DGLA of the form $ L = (L, d + [W, -], [-, -]) $. We show that under technical assumptions on the Maurer-Cartan element $ W $ the morphism $ π_W: L_W → (\H L)_{π^{\MC} (W)} $ is a quasi-isomorphism. As an application, we exhibit the case of the Hochschild complex of a minimal $ A_∞ $-category.

\subsection{Almost pronilpotent MC elements}
We start by fixing a suitable notion of filtration on $ L_∞ $-algebras. The idea is to create a notion of pronilpotent Maurer-Cartan elements such that embracing brackets and $ L_∞ $-morphisms by such elements is well-defined. We recall and adapt the following definition from \cite{Getzler-Linfty}. The definition is suited well for the Hochschild complex but does not suit for instance $ L_∞ $-algebras presented as infinite direct sums instead of direct products. In order to treat such cases, a simple adaption of the definition, for instance by means of a convergence requirement for the individual series, will suffice.

\begin{definition}
A \emph{filtered graded vector space} is a graded vector space with a complete decreasing filtration
\begin{equation*}
V = F^0 V \supseteq F^1 V \supseteq … \supseteq 0.
\end{equation*}
An \emph{almost pronilpotent $ L_∞ $-algebra} is an $ L_∞ $-algebra $ L $ with the structure of a filtered graded vector space $ \{F^k L\}_{k ∈ ℕ} $ such that for some constant $ C ∈ ℤ $ we have the rule
\begin{equation*}
l^m (F^{k_1} L × … × F^{k_m} L) ⊂ F^{k_1 + … + k_m + C - m} L.
\end{equation*}
An \emph{almost pronilpotent Maurer-Cartan element} is a Maurer-Cartan element $ W ∈ \MC(L) $ which lies in $ F^1 L $. An \emph{almost pronilpotent morphism} of $ L_∞ $-algebras is a morphism $ φ: L' → L $ such that for some constant $ D $ we have the rule
\begin{equation*}
φ^m (F^{k_1} L' × … × F^{k_m} L') ⊂ F^{k_1 + … + k_m + D - m} L.
\end{equation*}
\end{definition}

The following is a classical lemma in contexts with more rigid nilpotency assumptions, see for instance \cite[Proposition 4.4]{Sheng-Tang-Zhu}.

\begin{lemma}
Let $ L $ be an almost pronilpotent $ L_∞ $-algebra and $ W $ an almost pronilpotent Maurer-Cartan element. Then the following twisting provides a well-defined $ L_∞ $-structure on $ L $ which we denote by $ L_W $:
\begin{equation*}
l^k (a_1, …, a_k) = \sum_{i = 0}^∞ \frac{1}{i!} l^{k+i} (W, …, W, a_1, …, a_k).
\end{equation*}
\end{lemma}

Let us now examine the twisting of morphisms of almost pronilpotent $ L_∞ $-algebras.

\begin{lemma}
Let $ φ: L' → L $ be an almost pronilpotent morphism of $ L_∞ $-algebras. If $ W $ is an almost pronilpotent Maurer-Cartan element of $ L' $, then $ φ^{\MC} (W) $ is a well-defined almost pronilpotent Maurer-Cartan element of $ L $. The following assignment provides a well-defined almost pronilpotent $ L_∞ $-morphism $ φ_W: L'_W → L_{φ^{\MC} (W)} $:
\begin{equation*}
φ_W (a_k, …, a_1) = \sum_{i = 0}^∞ \frac{1}{i!} φ^{i+k} (W, …, W, a_k, …, a_1).
\end{equation*}
\end{lemma}

\begin{proof}
Since $ W $ is almost pronilpotent, the assignment is well-defined. Moreover, $ φ_W $ is skew-symmetric because $ φ $ is skew-symmetric. The remaining part is to verify the $ L_∞ $-morphism relations. Expand out the left-hand side of the $ L_∞ $-morphism relations as
\begin{align}
\label{eq:spectral-twistedLHS}
& \sum_{i = 0}^∞ \sum_{j = 0}^∞ \sum_{m = 0}^∞ \sum_{σ ∈ S_{m, k-m}} \frac{1}{i! j!} × \\
& \quad φ^{i+(k-m+1)} (W, …, W, {l'}^{j+m} (W, …, W, a_{σ(1)}, …, a_{σ(m)}), a_{σ(m+1)}, …, a_{σ(k)})
\end{align}
and the right-hand side of the $ L_∞ $-morphism relations as
\begin{align}
\label{eq:spectral-twistedRHS}
& \sum_{\substack{1 ≤ r ≤ k \\ i_1 + … + i_r = k}} ~ \sum_{\substack{x, u_1, …, u_x, v_1, …, v_r \\ u_1 + … + u_x + v_1 + … + v_r = i+j}} 
\frac{1}{u_1! … u_x! i_1! … i_r!} × \\
& \quad l\big(φ^{u_1} (W, …, W), …, φ^{u_x} (W, …, W), \\
& \quad ~~ φ^{i_1 + v_1} (W, …, W, a_{t(1)}, …, a_{t(i_1)}), …,
φ^{i_r + v_r} (W, …, W, a_{t(i_1 + … + i_{r-1} + 1)}, …, a_{t(k)})\big).
\end{align}
It is our task to compare these two expressions with the help of the existing $ L_∞ $-morphism relations which hold between $ l $, $ l' $ and $ φ $. Fix a number $ N ≥ 0 $ and regard the input sequence $ W, …, W, a_1, …, a_k $ of length $ N + k $. Evaluate the left-hand side of the $ L_∞ $-morphism relations for $ l, l', φ $. Then for every pair $ (i, j) $ with $ i+j = N $ the left-hand sides includes precisely $ {N \choose i} $ times the following term, because this is the number of possible choices for the first set of $ W $ insertions:
\begin{equation*}
φ^{i+(k-m+1)} (W, …, W, {l'}^{j+m} (W, …, W, a_{σ(1)}, …, a_{σ(m)}), a_{σ(m+1)}, …, a_{σ(k)}).
\end{equation*}
Evaluate now the right-hand side of the $ L_∞ $-morphism relations. This yields precisely $ {N \choose u_1, …, u_x, v_1, …, v_r} $ times the following term, because this is the number of possible choices for the partitioning of $ W $ insertions:
\begin{align*}
& l\big(φ^{u_1} (W, …, W), …, φ^{u_x} (W, …, W), \\
& \quad φ^{i_1 + v_1} (W, …, W, a_{t(1)}, …, a_{t(i_1)}), …, \\
& \quad φ^{i_r + v_r} (W, …, W, a_{t(i_1 + … + i_{r-1} + 1)}, …, a_{t(k)})\big).
\end{align*}
To conclude, the $ L_∞ $-relations for $ l, l', φ $ feature precisely $ i! j! {i+j \choose i} $ times the summand in \eqref{eq:spectral-twistedLHS} and $ u_1! … u_x! i_1! … i_r! {i+j \choose u_1, …, u_x, v_1, …, v_r} $ times the summand in \eqref{eq:spectral-twistedRHS}. Since both numbers agree, we conclude that the left-hand side and right-hand side of the $ L_∞ $-morphism relations for $ l'_W $, $ l_{φ^{\MC}} $ and $ φ_W $ match up exactly. There are non-matched terms in in \eqref{eq:spectral-twistedLHS} arising from the choice $ m = 0 $. These are covered by the Maurer-Cartan equation for $ W $. We omit sign checks and finish the proof.
\end{proof}

\subsection{Two-step procedure for minimal model}
\label{sec:spectral-spectral}
In this section, we explain the basic strategy of the present paper. We start with a DGLA presented as the twisting of a simpler DGLA by an almost pronilpotent Maurer-Cartan element, and provide a two-step procedure to compute its minimal model. For sake of language, use the following terminology:
\begin{definition}
A \emph{spectral DGLA} is a DGLA presented in the form $ L_W $, where $ L $ is an almost pronilpotent DGLA $ L $ and $ W ∈ \MC(L) $ is an almost pronilpotent Maurer-Cartan element.
\end{definition}

In the present paper, we exhibit a strategy to compute minimal models of spectral DGLAs. This strategy consists of the following steps:
\begin{enumerate}
\item Start from a spectral DGLA $ L_W $.
\item Compute the minimal model $ \H L $ together with the projection $ π: L → \H L $.
\item Compute the minimal model $ \H((\H L)_{π^{\MC} (W)}) $ together with the projection
\begin{equation*}
π': (\H L)_{π^{\MC} (W)} → \H((\H L)_{π^{\MC} (W)}).
\end{equation*}
\item Form the \emph{spectral projection}
\begin{equation*}
π' ∘ π_W: L_W \overset{π_W}{→} (\H L)_{π^{\MC} (W)} \overset{π'}{→} \H((\H L)_{π^{\MC} (W)}).
\end{equation*}
\item Investigate whether it is a quasi-isomorphism.
\end{enumerate}

Since we do not have an explicit expression for the projection $ π: L → \H L $ at our disposal, it is useful to consider the composition of $ L_∞ $-inclusion $ i: \H L → L $ as an alternative path. In these cases, the twisting element on $ \H L $ should be some element $ W_H ∈ \MC(\H L) $ such that $ W = i^{\MC} (W_H) $. The strategy can then be formulated as follows:

\begin{enumerate}
\item Start from a spectral DGLA $ L_W $.
\item Compute the minimal model $ \H L $ together with the inclusion $ i: \H L → L $.
\item Find an almost pronilpotent $ W_H ∈ \MC(\H L) $ such that $ W = i^{\MC} (W_H) $.
\item Compute the minimal model $ \H((\H L)_{W_H}) $ together with the inclusion
\begin{equation*}
i': \H((\H L)_{W_H}) → (\H L)_{W_H}.
\end{equation*}
\item Form the \emph{spectral inclusion}
\begin{equation*}
i_{W_H} ∘ i': \H((\H L)_{W_H}) \overset{i'}{→} (\H L)_{W_H} \overset{i_{W_H}}{→} L_W.
\end{equation*}
\item Investigate whether it is a quasi-isomorphism.
\end{enumerate}

\begin{remark}
\label{rem:spectral-pi-i}
In the case of the inclusion procedure we have $ π^{\MC} (W) = W_H $, without further computation. Indeed,
\begin{equation*}
π^{\MC} (W) = (π ∘ i)^{\MC} (W_H) = \id_{\H L}^{\MC} (W_H) = W_H.
\end{equation*}
\end{remark}

\subsection{Hochschild cohomology of $ A_∞ $-categories}
The inclusion procedure applies directly to the Hochschild DGLA of a minimal $ A_∞ $-category. The idea is to split the Hochschild differential as $ d = [μ^2, -] + [μ^{≥3}, -] $ and interpret the cochain $ μ^{≥3} $ as an almost pronilpotent Maurer-Cartan element.

\begin{lemma}
Let $ \cat C $ be a graded associative $ A_∞ $-category and $ L = \HC(\cat C) $. Denote by $ \HC^i (\cat C) ⊂ \HC (\cat C) $ the subspace of cochains of length $ i $. Then $ L $ is almost pronilpotent with the filtration $ F^k L = \prod_{i = k}^∞ \HC^i (\cat C) $. The minimal model $ \HH(\cat C) $ obtained from the Kadeishvili construction is almost pronilpotent with the filtration $ F^k \HH(\cat C) = \prod_{i = k}^∞ \HH^i (\cat C) $. The projection $ π: \HC(\cat C) → \HH(\cat C) $ and the inclusion $ i: \HH(\cat C) → \HC(\cat C) $ are almost pronilpotent as well.
\end{lemma}

\begin{proof}
To see that $ L $ is almost pronilpotent, it suffices to note that the Gerstenhaber bracket satisfies $ [ν, η] ∈ \HC^{i+j-1} (\cat C) $ whenever $ ν ∈ \HC(\cat C) $ and $ η ∈ \HC(\cat C) $. To see that $ \HH(\cat C) $ is almost pronilpotent, note that the $ L_∞ $-bracket $ l^k $ increases the total length of the input cochains by $ 1-k $. Similarly, the $ k $-adic components $ π^k $ and $ i^k $ increase the total length of the input cochains by $ -k $. This finishes the proof.
\end{proof}

\begin{lemma}
Let $ \cat C $ be a minimal $ A_∞ $-category and $ {\cat C}_2 $ the underlying associative category. Then $ W = μ^{≥3} $ is an almost pronilpotent Maurer-Cartan element of $ \HC({\cat C}_2) $, and $ \HC(\cat C) $ is spectral:
\begin{equation*}
\HC(\cat C) = \HC({\cat C}_2)_W.
\end{equation*}
\end{lemma}

\begin{proof}
The element $ W = μ^{≥3} $ is a Maurer-Cartan element of $ \HC({\cat C}_2) $ because
\begin{equation*}
[μ^2, μ^{≥3}] + \frac{1}{2} [μ^{≥3}, μ^{≥3}] = (μ^2 + μ^{≥3}) · (μ^2 + μ^{≥3}) = 0.
\end{equation*}
As a cochain, $ W $ is concentrated in length $ ≥ 3 $ and therefore almost pronilpotent. Observe that $ \HC({\cat C}_2)_W $ has differential $ [μ^2, -] + [μ^{≥3}, -] $ and bracket $ [-, -] $ and therefore agrees with $ \HC(\cat C) $, as desired.
\end{proof}

\section{Hochschild cohomology of gentle algebras}
\label{sec:gtl}
In this section, we compute the $ L_∞ $-structure on the Hochschild cohomology of $ \Gtl Σ $. The idea is to apply the spectral inclusion procedure of \autoref{sec:spectral-spectral}:
\begin{center}
\begin{tikzpicture}
\path (0, 0) node[align=center] {$ \GtlT Σ $ \\ Associative category};
\path (2, 0) node {\Large $ \rightsquigarrow $};
\path (4, 0) node[align=center] {$ \HH(\GtlT Σ) $ \\ Classical Hochschild};
\path (6, 0) node {\Large $ \rightsquigarrow $};
\path (8, 0) node [align=center] {$ \H(\HH(\GtlT Σ)_{W_H}) $ \\ $ A_∞ $-Hochschild};
\end{tikzpicture}
\end{center}
We start by regarding the associative category $ \GtlT Σ $ and computing part of the $ L_∞ $-inclusion $ i: \HH(\GtlT Σ) → \HC(\GtlT Σ) $. We pick a Maurer-Cartan element $ W_H ∈ \HH(\GtlT Σ) $ and prove that $ i^{\MC} (W_H) = μ^{≥3}_{\Gtl Σ} $. Then we compute the minimal model $ \H((\HH(\GtlT Σ))_{W_H}) $ and part of the $ L_∞ $-inclusion $ i': \H((\HH(\GtlT Σ))_{W_H}) → (\HH(\GtlT Σ))_{W_H} $. The main result \autoref{th:gtl-main} states that $ i_{W_H} $ is a quasi-isomorphism, which implies that $ \HH(\Gtl Σ) $ is formal and we possess a fairly explicit $ L_∞ $-quasi-isomorphism between $ \HC(\Gtl Σ) $ and its cohomology.

\subsection{Homological splitting of $ \HC(\GtlT Σ) $}
\label{sec:splitting1}
In this section, we choose a homological splitting for $ \HC(\GtlT Σ) $. The approach is to pick a few types of cochains to be included in $ R $, so that the computation of the minimal model becomes easier. We define a precise list $ R_{\Rlist} $ of these cochains and then show that this list of cochains can be extended to a full homological splitting for $ \HC(\GtlT Σ) $.

We start by defining the set of cochains $ R_{\Rlist} $. Let $ D $ and $ D' $ be disk sequences, given by interior angles $ α_1, …, α_l $ and $ β_1, …, β_m $. Let $ k, k' ≥ 1 $. We count indices of the two interior angle sequences $ α_1, …, α_l $ and $ β_1, …, β_m $ modulo $ l $ and $ m $, respectively.
\begin{itemize}
\item[A] Let $ i, j, t $ be any indices with $ β_i α_{j+t} ≠ 0 $ and $ γ $ or $ β $ be any empty or non-empty angles such that $ α_j γ ≠ 0 $ or $ β α_{j+1+kl} ≠ 0 $. Then we include the following cochains in $ R_{\Rlist} $:
\begin{align*}
ν(α_{j+kl-1}, …, α_{j+t+1}, β_{i+k'm-1}, …, β_i α_{j+t}, …, α_j γ) &= γ, \\
ν(β α_{j+kl-1}, …, α_{j+t+1}, β_{i+k'm-1}, …, β_i α_{j+t}, …, α_j) &= β.
\end{align*}
In the corner case $ t = 0 $, the angle $ γ $ is attached to $ β_i α_j $ instead of $ α_j $.
\item[B] Let $ i, j $ be any indices and $ β $ and $ γ $ be angles such that $ α_j β β_i γ ≠ 0 $. Then we include the following cochains in $ R_{\Rlist} $:
\begin{align*}
ν(α_{j+kl-1}, …, α_{j+1}, α_j β β_{i+k'm}, …, β_i γ) &= β β_i γ, \\
ν(β β_{i+k'm}, …, β_i γ α_{j+kl-1}, α_{j+kl-2}, …, α_j) &= β β_i γ.
\end{align*}
\item[C] Let $ i, j, t $ be any indices such that $ β_i α_{j+t} ≠ 0 $ or $ α_{j+t} β_{i+k'm-1} ≠ 0 $. Then we include the following cochains in $ R_{\Rlist} $, respectively:
\begin{align*}
ν(β α_{j+kl}, …, β_{i+k'm-1}, …, β_i α_{j+t}, …, α_j γ) &= β α_j γ, \\
ν(β α_{j+kl}, …, α_{j+t} β_{i+k'm-1}, …, β_i, …, α_j γ) &= β α_j γ.
\end{align*}
\item[D] Let $ i, j $ be any indices and $ β, γ $ be angles such that $ β_i α_j γ ≠ 0 $ and $ β α_j ≠ 0 $, or $ β α_{j+kl} β_{i+k'm-1} ≠ 0 $ and $ β α_j γ ≠ 0 $. Then we include the following cochains in $ R_{\Rlist} $, respectively:
\begin{align*}
ν(β α_{j+kl}, …, β_{i+k'm-1}, …, β_i α_j γ) &= β α_j γ, \\
ν(β α_{j+kl} β_{i+k'm-1}, …, β_i, …, α_j γ) &= β α_j γ.
\end{align*}
\item[E] Let $ j $ be an index such that $ α_j $ winds around puncture $ q $. Then we include the following cochain in $ R_{\Rlist} $:
\begin{align*}
ν(α_{j+kl-1}, …, α_{j+1}) &= ℓ_q^k α_j^{-1}.
\end{align*}
\item[F] Let $ j $ be an index such that $ α_j $ winds around puncture $ q $. Then we include the following two cochains in $ R_{\Rlist} $:
\begin{align*}
ν(α_{j+kl-1}, …, α_j γ) &= ℓ_q^{k'} γ, \\
ν(β α_{j+kl}, …, α_{j+1}) &= ℓ_q^{k'} β.
\end{align*}
\item[G] Let $ j $ be any index and $ β, γ $ be angles such that $ β α_{j-1} ≠ 0 $ or $ α_j γ ≠ 0 $. If the disk sequence $ D = (α_1, …, α_l) $ is an elementary polygon, then we require $ β $ and $ γ $ to be non-empty. We include the following two cochains in $ R_{\Rlist} $, respectively:
\begin{align*}
ν(β α_{j+kl-1}, …, α_j) &= β, \\
ν(α_{j+kl-1}, …, α_j γ) &= γ.
\end{align*}
\end{itemize}

\begin{definition}
The space $ R_{\Rpart} ⊂ \HC(\GtlT Σ) $ is defined as the completion space of $ R_{\Rlist} $ in the sense of \autoref{def:prelim-completion}.
\end{definition}

\begin{figure}
\centering
\begin{tikzpicture}
\begin{scope}[shift={(0, 1.5)}, scale=2]
\path[draw] (0, 0) -- ++(2, 0) -- ++(up:1) -- ++(left:2) -- ++(down:1);
\path[draw] (1, 0) -- ++(up:1);
\path[draw, ->, bend right=90, looseness=1.5] (1.35, 0) to (0.65, 0);
\path[draw, <-, bend left=90, looseness=1.5] (1.35, 1) to (0.65, 1);
\path[draw] (-0.1, 0.45) -- (0.1, 0.45);
\path[draw] (-0.1, 0.5) -- (0.1, 0.5);
\path (0.5, 0.5) node {$ D_1 $};
\path (1.5, 0.5) node {$ D_2 $};
\path[draw, ->, bend right=45] (0, 0.7) to node[pos=1, above] {} (0.3, 1);
\path[draw, <-, bend left=45] (0, 0.3) to node[pos=1, below] {} (0.3, 0);
\path (0.85, 0.9) node {$ α_s $};
\path (0.84, 0.1) node {$ α_{s+1} $};
\path (1.15, 0.9) node {$ β_{t+1} $};
\path (1.15, 0.1) node {$ β_t $};
\path (0.1, 0.9) node {$ α_1 $};
\path (0.1, 0.1) node {$ α_l $};
\end{scope}
\begin{scope}[shift={(8, 0)}, scale=2]
\path[draw] (0, 0) -- ++(2, 0) -- ++(up:1) -- ++(left:2) -- ++(down:1);
\path[draw] (1, 0) -- ++(up:1);
\path[draw, ->, bend right=45] (0.65, 1) to (1, 0.7);
\path[draw, ->, bend right=45] (1, 0.7) to (1.35, 1);
\path[draw, ->, bend right=90, looseness=1.5] (1.35, 0) to (0.65, 0);
\path[draw] (-0.1, 0.45) -- (0.1, 0.45);
\path[draw] (-0.1, 0.5) -- (0.1, 0.5);
\path (0.5, 0.5) node {$ D_1 $};
\path (1.5, 0.5) node {$ D_2 $};
\path[draw, ->, bend right=45] (0, 0.7) to node[pos=1, above] {\tiny } (0.35, 1);
\path[draw, <-, bend left=45] (0, 0.3) to node[pos=1, below] {\tiny } (0.3, 0);
\path (0.85, 0.9) node {$ α_s $};
\path (0.84, 0.1) node {$ α_{s+1} $};
\path (1.15, 0.9) node {$ β_{t+1} $};
\path (1.15, 0.1) node {$ β_t $};
\path (0.1, 0.9) node {$ α_1 $};
\path (0.1, 0.1) node {$ α_l $};

\end{scope}
\begin{scope}[shift={(8, 3)}, scale=2]
\path[draw] (0, 0) -- ++(2, 0) -- ++(up:1) -- ++(left:2) -- ++(down:1);
\path[draw] (1, 0) -- ++(up:1);
\path[draw, ->, bend right=45] (1.35, 0) to (1, 0.3);
\path[draw, ->, bend right=45] (1, 0.3) to (0.65, 0);
\path[draw, ->, bend right=90, looseness=1.5] (0.65, 1) to (1.35, 1);
\path[draw] (-0.1, 0.45) -- (0.1, 0.45);
\path[draw] (-0.1, 0.5) -- (0.1, 0.5);
\path (0.5, 0.5) node {$ D_1 $};
\path (1.5, 0.5) node {$ D_2 $};
\path[draw, ->, bend right=45] (0, 0.7) to node[pos=1, above] {\tiny } (0.35, 1);
\path[draw, <-, bend left=45] (0, 0.3) to node[pos=1, below] {\tiny } (0.3, 0);
\path (0.85, 0.9) node {$ α_s $};
\path (0.84, 0.1) node {$ α_{s+1} $};
\path (1.15, 0.9) node {$ β_{t+1} $};
\path (1.15, 0.1) node {$ β_t $};
\path (0.1, 0.9) node {$ α_1 $};
\path (0.1, 0.1) node {$ α_l $};

\end{scope}
\begin{scope}[shift={(-1.5, 0)}, scale=1.666666]
\path[draw, decorate, decoration={brace, amplitude=4pt}] (8.2, 3) to node[midway, shift={(0.1, 0)}, right] {$ ∈ R $} (8.2, 1.8);
\path[draw, decorate, decoration={brace, amplitude=4pt}] (9, 3) to node[midway, shift={(0.1, 0)}, right] {$ ∈ \Im(d) $} (9, 0);
\path (6.9, 1.5) node {\large $ + $};
\path[draw, <-] (3.5, 1.4) to node[midway, below] {$ h = $ “heal”} (5.5, 1.4);
\path[draw, ->] (3.5, 1.6) to node[midway, above] {$ d = $ “divide”} (5.5, 1.6);
\end{scope}
\end{tikzpicture}
\caption{This figure illustrates how the differential $ d $ and $ h $ act on certain types of cochains. The cochain $ ν $ illustrated on the left-hand side takes nonzero value on the illustrated non-elementary disk sequence $ α_1, …, β_{t+1} α_s, …, α_{s+1} β_t, …, α_l $. In consequence, its differential $ dν $ takes nonzero value on all possible ways to divide the disk sequence. As a useful mnemonic, we say that $ d $ “divides” the angle sequence and $ h $ “heals” the angle sequence.}
\label{fig:splitting1-dh}
\end{figure}
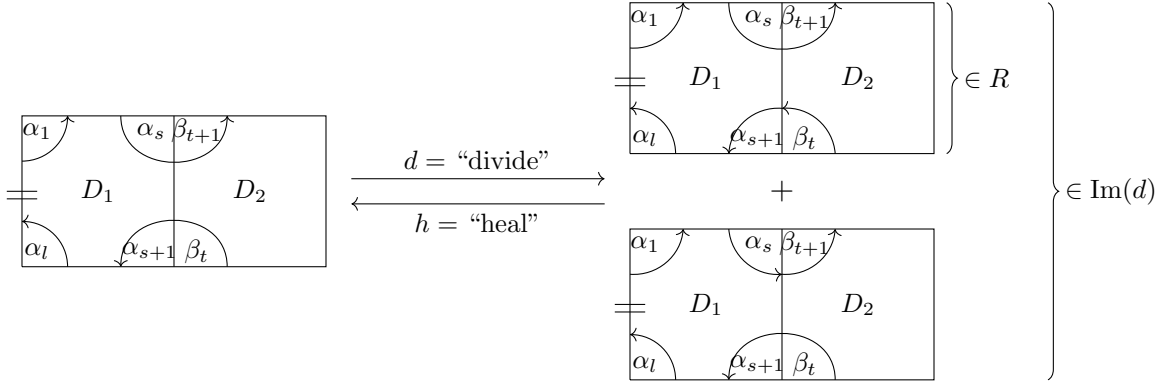

\begin{lemma}
\label{th:splitting1-RKer}
We have $ R_{\Rpart} ∩ \Ker(d) = 0 $.
\end{lemma}

\begin{proof}
Let $ ν ∈ R_{\Rpart} ∩ \Ker(d) $. We can assume that $ ν $ is a cochain of a specific length. In case $ ν $ is of length $ 0 $ or $ 1 $, it is an elementary check that $ ν = 0 $. Let us therefore assume that the length of $ ν $ is at least 2. It is our task to show that $ ν $ vanishes on all sequences of angles provided in the list $ R_{\Rlist} $. We proceed step by step, checking for every sequence of types A till H that $ ν $ vanishes on that sequence. We use the notation of $ R_{\Rlist} $, in particular the notation of indices $ i, j, t $ and angles $ β, γ $ and puncture $ q $. We shorten all signs to $ ± $ in order to enhance legibility.

\begin{itemize}
\item[A] The functionality of the differential is illustrated in \autoref{fig:splitting1-dh}. We have that
\begin{align*}
& (dν)(α_{j+kl-1}, …, α_{j+t+1}, β_{i+k'm-1}, …, β_i, α_{j+t}, …, α_j γ) \\
& \quad =
± ν(α_{j+kl-1}, …, α_{j+t+1}, β_{i+k'm-1}, …, β_i α_{j+t}, …, α_j γ) \\
& \qquad ±
ν(α_{j+kl-1}, …, α_{j+t+1} β_{i+k'm-1}, …, β_i, α_{j+t}, …, α_j γ), \\
& (dν)(β α_{j+kl-1}, …, α_{j+t+1}, β_{i+k'm-1}, …, β_i, α_{j+t}, …, α_j) \\
& \quad = ± ν(β α_{j+kl-1}, …, α_{j+t+1}, β_{i+k'm-1}, …, β_i α_{j+t}, …, α_j) \\
& \qquad ± ν(β α_{j+kl-1}, …, α_{j+t+1} β_{i+k'm-1}, …, β_i, α_{j+t}, …, α_j).
\end{align*}
It is our task to show that the first summands in both rows vanish. Since $ dν = 0 $, it suffices ot show that the second summands in both rows vanish. This is indeed the case, because the arc $ t(α_{j+t+1}) = h(β_{i+k'm-1}) $ uniquely cuts the union $ D ∪ D' $ into two disk sequences so that there is no way to interpret the angle sequence in the second summands as one of the types A till H.
\item[B] We have that
\begin{align*}
(dν)(α_{j+kl-1}, …, α_{j+1}, α_j, β β_{i+k'm}, …, β_i γ) &= ± ν(α_{j+kl-1}, …, α_{j+1}, α_j β β_{i+k'm}, …, β_i γ), \\
(dν)(β β_{i+k'm}, …, β_i, γ α_{j+kl-1}, α_{j+kl-2}, …, α_j) &= ± ν(β β_{i+k'm}, …, β_i γ α_{j+kl-1}, α_{j+kl-2}, …, α_j).
\end{align*}
We conclude that both right-hand sides vanish.
\item[C] We have that
\begin{align*}
& (dν)(β α_{j+kl}, …, β_{i+k'm-1}, …, β_i, α_{j+t}, …, α_j γ) \\
& \quad = ± ν(β α_{j+kl}, …, β_{i+k'm-1}, …, β_i α_{j+t}, …, α_j γ) \\
& \qquad ± ν(β α_{j+kl}, …, α_{j+t+1} β_{i+k'm-1}, …, β_i, α_{j+t}, …, α_j γ), \\
& (dν)(β α_{j+kl}, …, α_{j+t}, β_{i+k'm-1}, …, β_i, …, α_j γ) \\
& \quad = ± ν(β α_{j+kl}, …, α_{j+t} β_{i+k'm-1}, …, β_i, …, α_j γ) \\
& \qquad ± ν(β α_{j+kl}, …, α_{j+t}, β_{i+k'm-1}, …, β_i α_{j+t-1}, …, α_j γ).
\end{align*}
It is our task to show that the first summands in both rows vanish. Since $ dν = 0 $, it suffices ot show that the second summands in both rows vanish. This is indeed the case, for a reason similar to case A.
\item[D] We have that
\begin{align*}
& (dν)(β α_{j+kl}, …, β_{i+k'm-1}, …, β_i, α_j γ) \\
& \quad = ± ν(β α_{j+kl}, …, β_{i+k'm-1}, …, β_i α_j γ) \\
& \qquad ± ν(β α_{j+kl}, …, α_{j+1} β_{i+k'm-1}, …, β_i, α_j γ), \\
& (dν)(β α_{j+kl}, β_{i+k'm-1}, …, β_i, …, α_j γ) \\ 
& \quad = ± ν(β α_{j+kl} β_{i+k'm-1}, …, β_i, …, α_j γ) \\
& \qquad ± ν(β α_{j+kl}, β_{i+k'm-1}, …, β_i α_{j+kl-1}, …, α_j γ).
\end{align*}
In a fashion similar to the earlier cases, we conclude that the first summands on both rows vanish.
\item[E] We have that
\begin{align*}
(dν)(α_j, α_{j+kl-1}, …, α_{j+1}) &= ± α_j ν(α_{j+kl-1}, …, α_{j+1}) ± ν(α_j, α_{j+kl-1}, …, α_{j+2}) α_{j+1}.
\end{align*}
Let $ q $ and $ p $ denote the punctures around which $ α_j $ and $ α_{j+1} $ wind. Because $ ν ∈ R_{\Rpart} $, we can write $ ν(α_{j+kl-1}, …, α_{j+1}) = α_j^{-1} P(ℓ_q) $ and $ ν(α_j, α_{j+kl-1}, …, α_{j+2}) = α_{j+1}^{-1} Q(ℓ_p) $ where $ P $ and $ Q $ are univariate power series without constant term. The first summand then evaluates to $ P(ℓ_q) $, viewed as a linear combination of full turns starting at the arc end on which $ α_j $ ends. The second summand evaluates to $ Q(ℓ_p) $, viewed as a linear combination of full turns starting at the arc end at which $ α_{j+1} $ starts. Since these two arc ends are the opposite ends of the arc $ h(α_j) = t(α_{j+1}) $ and $ dν = 0 $, we conclude that both summands vanish.
\item[F] We have
\begin{align*}
(dν)(α_{j+kl-1}, …, α_j, γ) &= ± ν(α_{j+kl-1}, …, α_j γ) ± ν(α_{j+kl-1}, …, α_j) γ, \\
(dν)(β, α_{j+kl}, …, α_{j+1}) &= ± ν(β α_{j+kl}, …, α_{j+1}) ± β ν(α_{j+kl}, …, α_{j+1}).
\end{align*}
It is our task to show that $ ν(α_{j+kl-1}, …, α_j γ) $ does not contain a scalar multiple of $ γ $ and $ ν(β α_{j+kl}, …, α_{j+1}) $ does not contain a scalar multiple of $ β $. Assume first that the disk sequence $ D = (α_1, …, α_l) $ is an elementary polygon. Then by virtue of our choice to include only those cochains into type F with non-empty $ β $ and $ γ $, we have that $ ν(α_{j+kl-1}, …, α_j) $ and $ ν(α_{j+kl}, …, α_{j+1}) $ do not contain identities, finishing the argument. Assume now that the disk sequence $ D $ is not an elementary polygon. Then the angle sequence $ α_j, …, α_{j+kl-1} $ can be written in the form $ γ_1, …, δ_1 γ_s, …, γ_{s+1} δ_r, …, γ_t $ for some disk sequences $ γ_1, …, γ_t $ and $ δ_1, …, δ_r $. We have
\begin{equation*}
(dν)(γ_t, …, γ_{s+1} δ_r, …, δ_1, γ_s, …, γ_1) = ± ν(γ_t, …, γ_{s+1} δ_r, …, δ_1 γ_s, …, γ_1).
\end{equation*}
We conclude that $ ν(α_{j+kl-1}, …, α_j) $ vanishes. Analogous reasoning shows that $ ν(α_{j+kl}, …, α_{j+1}) $ vanishes. This finishes the argument.
\item[G] This case is similar to case F.
\end{itemize}
This finishes the proof.
\end{proof}

We are now ready to construct a homological splitting $ H ⊕ \Im(d) ⊕ R $ for the DGLA $ \HC(\GtlT Σ) $. Indeed, by virtue of \autoref{th:splitting1-RKer} there exists a choice of $ R $ such that $ R_{\Rpart} ⊂ R $.

\begin{definition}
\label{def:splitting1-def}
The cohomology space $ H = H^{\even} ⊕ H^{\odd} $ is defined by the following completions:
\begin{align*}
H^{\even} &= \overline{\{ν_λ\}_{λ ∈ \sporadicG}, \{ν_B\}_{B ∈ Σ_2}, \{ν_{q, k, \even}\}_{q ∈ M, k ≥ 1}, \{ν_{B, k, \even}\}_{B ∈ Σ_2, k ≥ 1}}, \\
H^{\odd} &= \overline{ν_1, \{ν_{q, k, \odd}\}_{q ∈ M, k ≥ 1}, \{ν_{B, k, \odd}\}_{B ∈ Σ_2}}.
\end{align*}
The complement space $ R $ is defined by any graded vector space complement of $ H ⊕ \Im(d) $ in $ \HC(\GtlT Σ) $ such that $ R_{\Rpart} ⊂ R $.
\end{definition}

\subsection{Higher structure on $ \HH(\GtlT Σ) $}
\label{sec:higher1}
In this section, we compute the $ L_∞ $-structure on $ \HH(\GtlT Σ) $. The strategy is to evaluate the brackets and higher brackets by means of trees with respect to the homological splitting from \autoref{def:splitting1-def}. We shall compute most brackets explicitly and omit a few less relevant brackets.

We start by examining the brackets of cochains $ ν_{B, k, \odd} $ associated with elementary polygons. It is conventient to perform this calculation more generally for brackets of auxiliary cochains associated with any disk sequences. We define auxiliary cochains as follows:

\begin{definition}
Let $ D $ be a disk sequence given by angles $ α_1, …, α_l $. Then denote by $ ν_D $ the cochain given by $ ν(β α_l, …, α_1) = β $ and $ ν(α_l, …, α_1 γ) = (-1)^{|γ|} γ $, for empty or non-empty angles $ β $ and $ γ $.
\end{definition}

Note that by definition we have that $ ν_{B, 1, \odd} = ν_B $ if $ B $ is an elementary polygon. Moreover, note that $ ν_D ∈ R $ if $ D $ is not an elementary polygon, in contrast to $ ν_{B, 1, \odd} ∈ H $.

Let us now examine the brackets of auxiliary cochains. Let $ D_1 $ and $ D_2 $ be disk sequences given by angles $ α_1, …, α_l $ and $ β_1, …, β_m $, respectively. Assume that $ D_1 $ and $ D_2 $ connect togother to a larger disk sequence $ D_1 ∪ D_2 $ along the angle pairs $ (α_s, β_{t+1}) $ and $ (α_{s+1}, β_t) $, as depicted in \autoref{fig:splitting1-dh}. The indices $ s, t $ are counted modulo $ l $ and $ m $, respectively. We denote by $ C(D_1, D_2) $ the set of all triples $ (s, t, D) $, consisting of connecting indices $ s, t $ together with the connected disk sequence $ D $. Simply speaking, \autoref{th:higher1-auxbracket} says that applying a bracket followed by the codifferential joins together two disk sequences. Conversely, applying the differential cuts larger disk sequences apart into brackets of shorter disk sequences. This phenomenon is illustrated in \autoref{fig:splitting1-dh}.

\begin{lemma}
\label{th:higher1-auxbracket}
Let $ D_1 $ and $ D_2 $ be disk sequences, then
\begin{equation*}
[ν_{D_1}, ν_{D_2}] = - \sum_{D ∈ C(D_1, D_2)} d(ν_D).
\end{equation*}
\end{lemma}

\begin{proof}
Let $ D_1 $ and $ D_2 $ be given by angles $ α_1, …, α_l $ and $ β_1, …, β_m $, respectively. It is our task to analyze the angle sequences on which the Gerstenhaber products $ ν_{D_1} · ν_{D_2} $ and $ ν_{D_2} · ν_{D_1} $ take nonzero value. Indeed, the product $ ν_{D_1} · ν_{D_2} $ only takes nonzero values on the following sequences, for every $ (s, t, D) ∈ C(D_1, D_2) $:
\begin{align*}
& α_{i+l-1}, …, α_{s+1}, β_{t+m}, …, β_{t+1} α_s, …, α_i γ, \\
& β α_{i+l-1}, …, α_{s+1}, β_{t+m}, …, β_{t+1} α_s, …, α_i, \\
& α_{i+l-1}, …, α_{s+2}, α_{s+1} β_{t+m}, …, β_{t+1}, α_s, …, α_i γ, \\
& β α_{i+l-1}, …, α_{s+2}, α_{s+1} β_{t+m}, …, β_{t+1}, α_s, …, α_i.
\end{align*}
With the notation $ u = ‖α_i‖ + … + ‖α_{s-1}‖ + |α_s| $, the precise values are:
\begin{align*}
(-1)^{‖α_i γ‖ + … + ‖α_{s-1}‖ + |α_s| + |γ|} γ &= (-1)^u γ, \\
(-1)^{‖α_i‖ + … + ‖α_{s-1}‖ + |α_s|} β &= (-1)^u β, \\
(-1)^{‖α_i γ‖ + … + ‖α_s‖ + |γ|} γ &= (-1)^{u+1} γ, \\
(-1)^{‖α_i‖ + … + ‖α_s‖} β &= (-1)^{u+1} β.
\end{align*}
The precise values arising from $ dν_D $ are:
\begin{align*}
(-1)^{1+1+‖α_i γ‖ + … + ‖α_{s-1}‖ + ‖β_{t+1} α_s‖ + … + ‖β_{t+m-1}‖ + |β_{t+m}| + |γ|} γ &= (-1)^{u+1} γ, \\
(-1)^{1+1+‖α_i‖+…+‖α_{s-1}‖+‖β_{t+1} α_s‖ + … + ‖β_{t+m-1}‖ + |β_{t+m}|} β &= (-1)^{u+1} β, \\
(-1)^{1+1+‖α_i γ‖ + … + ‖α_{s-1}‖ + |α_s| + |γ|} γ &= (-1)^u γ, \\
(-1)^{1+1+‖α_i‖ + … + ‖α_{s-1}‖ + |α_s|} β &= (-1)^u β.
\end{align*}
In the first two rows, we have used that $ ‖β_{t+1}‖ + … + ‖β_{t+m}‖ = 0 $ since $ D_2 $ is a disk sequence. The Gerstenhaber product $ ν_{D_2} · ν_{D_1} $ and the element $ dν_D $ take similar values with $ D_1 $ and $ D_2 $ swapped. We conclude that the sum of the cochains $ - d(ν_D) $ takes the exact same values as $ [ν_{D_1}, ν_{D_2}] $, which finishes the proof.
\end{proof}

\begin{lemma}
\label{th:higher1-brackets}
Let $ B, B' $ be elementary polygons, $ q, q' $ be punctures, $ k, k' ≥ 1 $ and $ ν_κ, ν_λ $ be any generic sporadic or polygon sporadic classes. Then in $ \HH(\GtlT Σ) $, we have the following brackets of cohomology basis elements.
\begin{align*}
[ν_{q, k, \odd}, ν_{q', k', \odd}] &= 0, \\
[ν_{q, k, \even}, ν_{q', k', \odd}] &= δ_{q,q'} ν_{q, k+k', \odd}, \\
[ν_{q, k, \even}, ν_{q', k', \even}] &= δ_{q,q'} ν_{q, k+k', \even}, \\
[ν_λ, ν_{q, k, \odd}] &= k λ_{ℓ_q} ν_{q, k, \odd}, \\
[ν_λ, ν_{q, k, \even}] &= k λ_{ℓ_q} ν_{q, k, \even}, \\
[ν_κ, ν_λ] &= 0, \\
[ν_{B, k, \odd}, ν_{B', k', \odd}] &= 0, \\
[ν_{B, k, \odd}, ν_{B', k', \even}] &= δ_{B, B'} ν_{B, k+k', \odd}, \\
[ν_{B, k, \odd}, ν_{q, k', \odd}] &= 0, \\
[ν_{B, k, \odd}, ν_{q, k', \even}] &= 0, \\
[ν_{B, k, \odd}, ν_λ] &= k \sum_{i = 1}^l λ(β_i) ν_{B, k, \odd}.
\end{align*}
All higher brackets vanish. Moreover, if $ a = \sum_{i = 0}^∞ a_i $ and $ b = \sum_{j = 0}^∞ b_j $ are two converging series in $ \HH(\GtlT Σ) $, then
\begin{equation*}
[a, b] = \sum_{N = 0}^∞ \sum_{i+j = N} [a_i, b_j].
\end{equation*}
\end{lemma}

\begin{proof}
We check all individual brackets in sequence. After that, we comment on brackets of series as well as the higher brackets.

\begin{itemize}
\item Regard $ [ν_{q, k, \odd}, ν_{q', k', \odd}] $. It is immediate that the bracket vanishes in $ \HC(\GtlT Σ) $, therefore also vanishes in cohomology.

\item Regard $ [ν_{q, k, \even}, ν_{q', k', \odd}] $. In $ \HC(\GtlT Σ) $, this bracket has only a 0-adic component. We evaluate
\begin{align*}
[ν_{q, k, \even}, ν_{q', k', \odd}]^0 &= ν_{q, k, \even} (ν_{q', k', \odd}) = δ_{q, q'} ℓ_q^k ν_{q', k', \odd}^0 = δ_{q, q'} ν_{q, k + k', \odd}^0.
\end{align*}
Thus, the bracket equals $ ν_{q, k+k', \odd} $ which is already a cohomology class.

\item Regard $ [ν_{q, k, \even}, ν_{q', k', \even}] $. In $ \HC(\GtlT Σ) $, this bracket has only a 1-adic component and this component is a derivation. In case $ q = q' $, then for any indecomposable angle $ α $ winding around $ q $ we have
\begin{align*}
[ν_{q, k, \even}, ν_{q', k', \even}] (α) &= ν_{q, k, \even} (ν_{q, k', \even} (α)) - ν_{q, k', \even} (ν_{q, k, \even} (α)) \\
&= ν_{q, k, \even} (δ_{α, α_1} ℓ^{k'} α) - ν_{q, k', \even} (δ_{α, α_1} ℓ^k α) \\
&= δ_{α, α_1} \big(k' ℓ^{k'+k} α + ℓ^{k'+k} δ_{α_1, α} α - k ℓ^{k+k'} α - ℓ^{k+k'} δ_{α_1, α} α\big) \\
&= δ_{α, α_1} (k' - k) ℓ^{k+k'} α.
\end{align*}
We have used that $ ν_{q, k, \even} (ℓ_q) = ν_{q, k, \even} (α_s … α_1) = ℓ^{k+1}_q $ and $ ν_{q, k, \even} (ℓ_q^i) = i ℓ^{k+i}_q $. We conclude that $ [ν_{q, k, \even}, ν_{q', k', \even}] $ is precisely cohomology element $ (k' - k) ν_{q, k+k', \even} $.

\item Regard $ [ν_λ, ν_{q, k, \odd}] $. In $ \HC(\GtlT Σ) $, this bracket has only a 0-adic component given by
\begin{equation*}
[ν_λ, ν_{q, k, \odd}]^0 = ν_λ^1 (ν_{q, k, \odd}^0) = k λ (ℓ_q) ℓ_q^k.
\end{equation*}

\item Regard $ [ν_λ, ν_{q, k, \even}] $. In $ \HC(\GtlT Σ) $, this bracket has only a 1-adic component and it is a derivation. For any indecomposable angle $ α $ winding around $ q $, we get
\begin{align*}
[ν_λ, ν_{q, k, \even}]^1 (α) &=  ν_λ (δ_{α, α_1} ℓ_q^k α) - ν_{q, k, \even} (ν_λ (α)) \\
& = δ_{α, α_1} (k λ_{ℓ_q} α + ℓ_q λ_α) - δ_{α, α_1} ℓ_q^k (λ_α α) = δ_{α, α_1} k λ_{ℓ_q} α.
\end{align*}
The bracket vanishes on all indecomposable angles not winding around $ q $. We conclude that the bracket is precisely the cohomology element $ k λ_{ℓ_q} ν_{q, k, \even} $.

\item Regard $ [ν_κ, ν_λ] $. In $ \HC(\GtlT Σ) $, this bracket has only a 1-adic componend and it is a derivation. For any indecomposable angle $ α $ we have
\begin{equation*}
[ν_κ, ν_λ]^1 (α) = κ_α λ_α α - λ_α κ_α α = 0.
\end{equation*}
This shows that $ [ν_κ, ν_λ] $ vanishes in $ \HC(\GtlT Σ) $.

\item Regard $  [ν_{B, k, \odd}, ν_{B', k', \odd}] $. This case is very similar to \autoref{th:higher1-auxbracket}.

\item Regard $ η = [ν_{B, k, \odd}, ν_{B', k', \even}] $. Let $ α_1, … α_l $ be the interior angles of $ B $ and $ β_1, …, β_m $ the interior angles of $ B' $. Evaluating the Gerstenhaber bracket yields only nonzero values on the following angle sequences:
\begin{align*}
ν_{B, k, \odd} (α_{j+kl-1}, …, α_{j+t+1}, ν_{B', k', \even} (β_{i+k'm}, …, β_i), α_{j+t-1}, …, α_j γ), \\
ν_{B, k, \odd} (β α_{j+kl-1}, …, α_{j+t+1}, ν_{B', k', \even} (β_{i+k'm}, …, β_i), α_{j+t-1}, …, α_j), \\ 
ν_{B, k, \odd} (α_{j+kl-1}, …, α_{j+1}, ν_{B', k', \even} (α_j β β_{i+k'm}, …, β_i γ)), \\
ν_{B, k, \odd} (ν_{B', k', \even} (β β_{i+k'm}, …, β_i γ α_{j+kl-1}), α_{j+kl-2}, …, α_j), \\
%
ν_{B, k, \even} (β α_{j+kl}, …, ν_{B', k', \odd} (β_{i+k'm-1}, …, β_i α_{j+t}), …, α_j γ), \\
ν_{B, k, \even} (β α_{j+kl}, …, ν_{B', k', \odd} (α_{j+t} β_{i+k'm-1}, …, β_i), …, α_j γ), \\
ν_{B, k, \even} (β α_{j+kl}, …, ν_{B', k', \odd} (β_{i+k'm-1}, …, β_i α_j γ)), \\
ν_{B, k, \even} (ν_{B', k', \odd} (β α_{j+kl} β_{i+k'm-1}, …, β_i), …, α_j γ).
\end{align*}
The first and second row appear whenever coincidentally for some indices $ i, j, t $ we have $ β_i = α_{j+t} $. The third and fourth row appear whenever coinicidentally for some indices $ i, j $ and some angles $ β $ and $ γ $ we have $ α_j β β_i γ ≠ 0 $. The fifth and sixth row appear whenever coincidentally for some indices $ i, j, t $ we have $ β_i α_{j+t} ≠ 0 $ or $ α_{j+t} β_{i+k'm-1} ≠ 0 $, respectively. The seventh and eighth row appear whenever for coincidentally for some indices $ i, j $ and some angles $ β, γ $ we have $ β_i α_j γ ≠ 0 $ and $ β α_j ≠ 0 $, or $ β α_{j+kl} β_{i+k'm-1} ≠ 0 $ and $ β α_j γ ≠ 0 $, respectively.

The precise values of the terms including their signs from the Gerstenhaber bracket are the following:
\begin{align*}
(-1)^{|γ|} γ, \\
β, \\
(-1)^{|β β_i γ|} β β_i γ, \\
β β_i γ, \\
(-1)^{1+1+‖α_j γ‖ + … + ‖α_{j+t-1}‖ + |α_{j+t}| + |γ|} β α_j γ, \\
(-1)^{1+1+‖α_j γ‖ + … + ‖α_{j+t-1}‖ + |γ|} β α_j γ, \\
(-1)^{1+1+|α_j γ|} β α_j γ, \\
(-1)^{1+1+‖α_j γ‖ + … + ‖α_{j+kl-1}‖ + |γ|} γ.
\end{align*}

Let us draw the conclusion. The terms in the third, fourth, fifth, sixth, seventh and eighth row all lie in $ R $. If $ B = B' $, then up to terms from $ R $ we have $ [ν_{B, k, \odd}, ν_{B, k', \even}] = ν_{B, k+k', \odd} $. If $ B ≠ B' $, then $ [ν_{B, k, \odd}, ν_{B, k', \even}] $ lies in $ R $ and therefore vanishes in cohomology. This finishes the present part of the proof.

\item Regard $ η = [ν_{B, k, \odd}, ν_{q, k', \odd}] $. Let $ β_1, …, β_l $ be the interior angles of $ B $. Since $ ν_{q, k', \odd} $ has only a 0-adic component, the bracket only takes the following angle sequences:
\begin{align*}
ν_{B, k, \odd} (β_{j+kl-1}, …, β_{j+1}, ν_{q, k', \odd}^0) &= ± β_j^{-1} ℓ_q^{k'}.
ν_{B, k, \odd} (ν_{q, k', \odd}^0, β_{j+kl-1}, …, β_{j+1}) &= ℓ_q^{k'} β_j^{-1}.
\end{align*}
These two rows appear whenever for some index $ j $ the angle $ β_j $ winds around $ q $. The inverse notation $ ℓ_q^{k'} β^{-1} $ denotes shortening of the $ k' $-fold turn around $ q $ by the angle $ β $. Both cochains lie in $ R $ and therefore the bracket vanishes.

\item Regard $ [ν_{B, k, \odd}, ν_{q, k', \even}] $. Let $ β_1, …, β_l $ be the interior angles of $ B $. Since $ ν_{q, k', \even} $ has only a 1-adic component, which strictly lengthens its input angle by a scalar multiple of $ k' $ full turns, the bracket only takes the following angle sequences:
\begin{align*}
ν_{B, k, \odd} (β_{j+kl-1}, …, ν_{q, k', \even} (β_j γ)) &= (-1)^{|γ|} \#ν_{q, k', \even} (β_j γ) ℓ_q^{k'} γ, \\
ν_{B, k, \odd} (ν_{q, k', \even} (β β_{j+kl}), …, β_{j+1}) &= \#ν_{q, k', \even} (β β_j) ℓ_q^{k'} β, \\
ν_{q, k', \even} (ν_{B, k, \odd} (β_{j+kl-1}, …, β_j γ)) &= (-1)^{|γ|} \#ν_{q, k', \even} (γ) ℓ_q^{k'} γ, \\
ν_{q, k', \even} (ν_{B, k, \odd} (β β_{j+kl}, …, β_{j+1})) &= \#ν_{q, k', \even} (γ) ℓ_q^{k'} β.
\end{align*}
These first and second rows appear whenever for some index $ j $ the angle $ β_j $ winds around $ q $. The third and fourth row appear whenever $ j $ is an index such that $ β_j $ winds around $ q $. We have used the notation $ ν_{q, k', \even} (β) = \# ν_{q, k', \even} (β) β ℓ_q^{k'} $ for any angle $ β $ which winds around $ q $. We conclude that all cochains lie in $ R $ and therefore the bracket vanishes.

\item Regard $ η = [ν_{B, k, \odd}, ν_λ] $. Let $ β_1, …, β_l $ be the interior angles of $ B $. Since $ ν_λ $ has only a 1-adic component, the bracket only takes the following angle sequences:
\begin{align*}
ν_{B, k, \odd} (β β_{j+kl-1}, …, ν_λ (β_i), …, β_j) &= λ(β_i) β, \\
ν_{B, k, \odd} (β_{j+kl-1}, …, ν_λ (β_i), …, β_j γ) &= (-1)^{|γ|} λ(β_i) γ, \\
ν_{B, k, \odd} (ν_λ (β β_{j+kl-1}), …, β_j) &= λ(β β_{j-1}) β, \\
ν_{B, k, \odd} (β_{j+kl}, …, ν_λ (β_{j+kl-1} γ)) &= (-1)^{|γ|} λ(β_{j-1} γ) γ, \\
ν_λ (ν_{B, k, \odd} (β β_{j+kl-1}, …, β_j)) &= λ(β) β, \\
ν_λ (ν_{B, k, \odd} (β_{j+kl-1}, …, β_j γ)) &= (-1)^{|γ|} λ(γ) γ.
\end{align*}
Summing up the terms, using additivity of $ λ $ as in $ λ(β β_{j-1}) = λ(β) + λ(β_{j-1}) $, we arrive at
\begin{align*}
η (β β_{j+kl-1}, …, β_j) = k \sum_{i = 1}^l λ(β_i) β, \\
η (β_{j+kl-1}, …, β_j γ) = k \sum_{i = 1}^l λ(β_i) (-1)^{|γ|} γ.
\end{align*}
We conclude $ η = k \sum_{i = 1}^l λ(β_i) ν_{B, k, \odd} $.
\end{itemize}
To wrap up the proof, let us comment on the brackets of converging series. Indeed, the projection $ π $, differential $ d $ and codifferential $ h $ map a cochain of length $ l $ to a cochain of length in the range $ l ± 1 $. Therefore the arguments in the present proof sum up in the case of infinite series as well.

Regarding the higher brackets, we make the following remarks. The higher brackets can be evaluated by means of Kadeishvili trees with inputs $ h_1, …, h_k ∈ H $. Every non-leaf non-root node is decorated by $ h [-, -] $, and the root is decorated by $ π [-, -] $. We immediately observe that the only nonvanishing nodes with two cohomology inputs are of the form $ h [ν_{B, k, \odd}, ν_{B', k', \odd}] $. By \autoref{th:higher1-auxbracket}, the result is a sum of auxiliary cochains $ ν_D $, which behave mostly in the same way as the set of cochains $ ν_{B, k, \odd} $. In particular, the only nonvanishing non-leaf nodes in the tree are linear combinations of auxiliary cochains $ ν_D $. The root of the tree therefore carries only projections of the form $ π [ν_D, ν_{D'}] $ which vanish thanks to \autoref{th:higher1-auxbracket}. This finishes the proof.
\end{proof}


Next we determine the Maurer-Cartan element $ W_H $ in $ \HH(\GtlT Σ) $ which corresponds to the Maurer-Cartan element $ W = μ^{≥3} $ in $ \HC(\GtlT Σ) $. The first step is to regard the inclusion $ L_∞ $-morphism $ i: \HH(\GtlT Σ) → \HC(\GtlT Σ) $. This morphism is explicitly computable by means of trees, but we shall not need the entire datum of $ i $ for our present goals. We shall make an educated guess for the Maurer-Cartan element $ W_H $ and define
\begin{equation*}
W_H = \sum_{B ∈ Σ_2} ν_{B, 1, \odd}.
\end{equation*}

\begin{lemma}
\label{th:higher1-MCpush}
$ W_H ∈ \HH(\GtlT Σ) $ is an almost pronilpotent Maurer-Cartan. We have
\begin{equation*}
i^{\MC} (W_H) = μ^{≥3}, \quad π^{\MC} (μ^{≥3}) = W_H.
\end{equation*}
\end{lemma}

\begin{proof}
We divide the proof into three steps. In the first step, we explain the strategy for computing $ i^{\MC} (W_H) $ and the trees to be evaluated. In the second step, we perform this evaluation as far as necessary. In the third step, we draw the conclusion.

For the first step, recall from \autoref{sec:prelim} that $ i^{\MC} (W_H) $ is given by the sum of weighted expressions obtained from trees with any number of inputs $ k ≥ 1 $ and with the leaves decorated by cochains $ ν_{B_1, 1, \odd}, …, ν_{B_k, 1, \odd} $ and $ h [-, -]_{\HC(\GtlT Σ)} $ on each node of the tree. In the remainder of the proof, we evaluate these trees and the resulting expressions. For temporary purposes, let us split $ μ^{≥3} = μ^{≥3}_1 + μ^{≥3}_{≥2} $ where $ μ^{≥3}_1 = \sum_{B ∈ Σ_2} ν_{B, 1, \odd} $ contains all cochains which take elementary polygons as inputs and $ μ^{≥3}_{≥2} $ contains all cochains which take all non-elementary disk sequences as inputs. We note that $ i^1 (W_H) = μ^{≥3}_1 $.

For the second part, we show that the trees with at least 2 leaves yield only cochains which, up to sign and weight factor, already lie in $ μ^{≥3}_{≥2} $. Indeed, thanks to \autoref{th:higher1-auxbracket} we have that $ [ν_{D_1}, ν_{D_2}] = ν_{D_1 ∪ D_2} $. Now let $ B_1, …, B_k $ be a sequence of $ k ≥ 2 $ elementary polygons. Then any result component of the trees with inputs $ ν_{B_1, 1, \odd}, …, ν_{B_k, 1, \odd} $ is a scalar multiple of a cochain already present in $ μ^{≥3}_{≥2} $.

For the third part, we claim that $ i^{\MC} (W_H) = μ^{≥3} $. Indeed, both $ i^{\MC} (W_H) $ and $ μ^{≥3} $ define Maurer-Cartan elements in $ \HC(\GtlT Σ) $, in other words $ A_∞ $-structures on $ \GtlT Σ $. Since they agree on elementary polygons and only take values on disk sequences, instead of arbitrary angle sequences, we conclude that $ i^{\MC} (W_H) = μ^{≥3} $. Thanks to \autoref{rem:spectral-pi-i} we conclude $ π^{\MC} (W) = W_H $ as well, which finishes the proof.
\end{proof}

\subsection{Computation of $ \H(\HH(\GtlT Σ)_{W_H}) $}
In this section, we compute the minimal model of $ \HH(\GtlT Σ)_{W_H} $. We start by describing a homological splitting of $ \HH(\GtlT Σ)_{W_H} $. Next, we compute the $ L_∞ $-structure on $ \H(\HH(\GtlT Σ)_{W_H}) $. As it turns out, the cohomology is simply an $ L_∞ $-subalgebra of $ \HH(\GtlT Σ) $ which only contains the cohomology classes $ ν_1 $, $ ν_{k, q, \odd} $, $ ν_{k, q, \even} $ and $ ν_λ $.

We start by recalling that $ \HH(\GtlT Σ) $ is a DGLA with zero differential and explicitly provided brackets. The sum $ W_H = \sum_{B ∈ Σ_2} ν_{B, k, \odd} $ is a Maurer-Cartan element on $ \HH(\GtlT Σ) $, and the DGLA $ \HH(\GtlT Σ)_{W_H} $ has the same underlying graded vector space as $ \HH(\GtlT Σ) $ and the same bracket but has differential given by $ d = [W_H, -] $. We are now ready to construct a homological splitting $ H ⊕ \Im(d) ⊕ R $ for $ \HH(\GtlT Σ)_{W_H} $.

\begin{definition}
The cohomology space $ H $ is the span of the following set of cochains:
\begin{itemize}
\item $ ν_1 $.
\item $ ν_{q, k, \odd} $ for every puncture $ q $ and $ k ≥ 1 $.
\item $ ν_{q, k, \even} $ for every puncture $ q $ and $ k ≥ 1 $.
\item $ ν_λ $ for every $ λ ∈ \sporadicG $.
\end{itemize}
The complement space $ R $ is the span of the following set of cochains:
\begin{itemize}
\item $ ν_{B, k, \even} $ for every elementary polygon $ B $ and $ k ≥ 1 $.
\item $ ν_λ $ for $ λ ∈ \sporadicP $.
\end{itemize}
\end{definition}

\begin{lemma}
The spaces $ H $ and $ R $ provides a homological splitting $ H ⊕ \Im(d) ⊕ R $ for $ \HH(\GtlT Σ)_{W_H} $.
\end{lemma}

\begin{proof}
Let us explain that $ \HH(\GtlT Σ)_{W_H} $ is the sum of the three spaces. Let $ B $ be an elementary polygon, and denote by $ B_0 $ the elementary polygon excluded from the polygon sporadic classes $ \sporadicP $. In case $ B ≠ B_0 $, we have $ ν_{B, 1, \odd} = d(ν_λ) $ where $ λ ∈ \sporadicP $ is the sporadic index associated with the polygon $ B $. In case $ B = B_0 $, we have $ ν_{B, k, \odd} = \sum_{λ ∈ \sporadicP, λ ≠ B_0} d(ν_λ) $. Furthermore we have $ ν_{B, k, \odd} = d(ν_{B, k-1, \even}) $ for $ k ≥ 2 $.

Next, we observe that $ H ⊂ \Ker(d) $ and the sum $ H ⊕ \Im(d) $ is direct. To see that $ R ∩ \Ker(d) = 0 $, regard the differential of any element of $ R $:
\begin{equation*}
d\big(\sum_{B ∈ Σ_2, k ≥ 1} c_{B, k} ν_{B, k, \even} + \sum_{λ ∈ \sporadicP} c_λ ν_λ\big) = \sum_{B ∈ Σ_2, k ≥ 1} c_{B, k} ν_{B, k+1, \odd} + \sum_{B ∈ Σ_2, B ≠ B_0} c_B ν_{B, 1, \odd}.
\end{equation*}
If this differential vanishes, then the individual scalars $ c_{B, k} $ and $ c_B $ vanish. Therefore $ R ∩ \Ker(d) = 0 $. This finishes the proof.
\end{proof}

\begin{proposition}
\label{th:gtl-H2-brackets}
In $ \H(\HH(\GtlT Σ)_{W_H}) $ we have the following brackets. All higher brackets vanish.
\begin{align*}
[ν_{q, k, \odd}, ν_{q', k', \odd}] &= 0, \\
[ν_{q, k, \even}, ν_{q', k', \odd}] &= δ_{q,q'} ν_{q, k+k', \odd}, \\
[ν_{q, k, \even}, ν_{q', k', \even}] &= δ_{q,q'} ν_{q, k+k', \even}, \\
[ν_λ, ν_{q, k, \odd}] &= k λ_{ℓ_q} ν_{q, k, \odd}, \\
[ν_λ, ν_{q, k, \even}] &= k λ_{ℓ_q} ν_{q, k, \even}, \\
[ν_κ, ν_λ] &= 0.
\end{align*}
\end{proposition}

\begin{proof}
By definition, each of these brackets is given by the projection to $ H $ of the result of the bracket in $ \HH(\GtlT Σ)_{W_H} $. We have computed the brackets in $ \HH(\GtlT Σ) $ in \autoref{th:higher1-brackets}, and each of them already lies in the cohomology space $ H $. In particular, their projection to $ H $ agrees with the values from \autoref{th:higher1-brackets} and any application of the codifferential $ h $ evaluates to zero. This finishes the proof.
\end{proof}

\subsection{Main result}
In this section, we construct the spectral inclusion
\begin{equation*}
\H(\HH(\GtlT Σ)_{W_H}) \overset{i'}{→} \HH(\GtlT Σ)_{W_H} \overset{i_{W_H}}{→} \HC(\GtlT Σ)_{μ^{≥3}} = \HC(\Gtl Σ).
\end{equation*}
We prove that it is an $ L_∞ $-quasi-isomorphism and explain how it transports Maurer-Cartan elements of $ \H(\HH(\GtlT Σ)_{W_H}) $ to deformations of $ \Gtl Σ $.

\begin{remark}
For general $ A_∞ $-categories $ \cat C $, it is unclear whether the map $ i_{W_H}: \H(\HH({\cat C}_2)_{W_H} → \HH({\cat C}_2)_{W_H} $ is a quasi-isomorphism. For gentle algebras, we are however able to check this by hand with the help of the description of Hochschild cohomology from \cite{Paper-I}. In result, the description of the entire $ L_∞ $-structure on $ \HH(\Gtl Σ) $ including the vanishing of the higher brackets is immediate. Moreover, we are able to describe the spectral inclusion $ i_{W_H} ∘ i' $ neatly on the level of Maurer-Cartan elements.
\end{remark}

\begin{remark}
In order to avoid confusion, we shall use the following parameter space $ V $ which does not reference Hochschild cohomology.
\begin{equation*}
V = \frac{ℂ[ℓ_q \running q ∈ M]}{(ℓ_q ℓ_p \running q ≠ p)}.
\end{equation*}
A parameter $ r ∈ V $ can alternatively be written as
\begin{equation*}
r = r_1 + \sum_{q ∈ M, k ≥ 1} r_{q, k} ℓ_q^k, \qquad r = r_1 + \sum_{q ∈ M} R_q (ℓ_q).
\end{equation*}
For each $ q ∈ M $ the series $ R_q (ℓ_q) $ is a polynomial.
\end{remark}

We shall now recall the specific deformations from \cite{Paper-I}. Let $ (B, \mathfrak{m}) $ be a deformation base and let $ r ∈ \mathfrak{m} \htensor V $ be a deformation parameter. Then the construction of \cite{Paper-I} provides an curved $ A_∞ $-deformation $ {}^r μ $ of $ μ $. The deformed curved $ A_∞ $-structure $ {}^r μ $ on $ \GtlT Σ $ is constructed by counting elementary polygons and in addition orbigons which are weighted with $ r_{q, k} $ factors. The curvature of each arc incident at a puncture $ q $ is given by $ R_q (ℓ_q) $, a linear combination of full turns. We can view the difference $ {}^r μ - μ ∈ \mathfrak{m} \htensor \HC(\Gtl Σ) $ as a Maurer-Cartan element of the DGLA $ \mathfrak{m} \htensor \HC(\Gtl Σ) $.

Thanks to \autoref{th:gtl-H2-brackets}, we have an identification between $ V $ and $ \H(\HH(\GtlT Σ)_{W_H})^{\odd} $. In particular, every $ r ∈ V $ can be viewed as an element of the DGLA $ \H(\HH(\GtlT Σ)_{W_H}) $ and every $ r ∈ \mathfrak{m} \htensor V $ can be viewed as a Maurer-Cartan element of the DGLA $ \mathfrak{m} \htensor \H(\HH(\GtlT Σ)_{W_H}) $.

With these preparations in mind, we can state our main theorem.

\begin{theorem}
\label{th:gtl-main}
Let $ (Σ, M) $ be a punctured surface with [NMD] and let $ \Gtl Σ $ be the $ A_∞ $-gentle algebra defined by means of a full arc system. Let $ (B, \mathfrak{m}) $ be a deformation base.
\begin{enumerate}
\item The spectral inclusion $ i_{W_H} ∘ i' $ is a quasi-isomorphism:
\begin{equation*}
i_{W_H} ∘ i': \H(\HH(\GtlT Σ)_{W_H}) \overset{i'}{→} \HH(\GtlT Σ)_{W_H} \overset{i_{W_H}}{→} \HC(\GtlT Σ)_{μ^{≥3}} = \HC(\Gtl Σ).
\end{equation*}
\item The induced morphism $ i_{W_H}^B ∘ {i'}^B $ is a quasi-isomorphism:
\begin{equation*}
i_{W_H}^B ∘ {i'}^B: \mathfrak{m} \htensor \H(\HH(\GtlT Σ)_{W_H}) \overset{{i'}^B}{→} \mathfrak{m} \htensor \HH(\GtlT Σ)_{W_H} \overset{i_{W_H}^B}{→} \mathfrak{m} \htensor \HC(\GtlT Σ)_{μ^{≥3}} = \mathfrak{m} \htensor \HC(\Gtl Σ).
\end{equation*}
\item On the level of Maurer-Cartan elements, $ i_{W_H}^B ∘ {i'}^B $ sends a parameter $ r ∈ \mathfrak{m} \htensor V $ to $ {}^r μ - μ $.
\end{enumerate}
\end{theorem}

\begin{proof}
We address the three statements in their natural order. For the first statement, note that the inclusion map $ (i')^1 $ sends any cohomology basis element $ ν ∈ \H(\HH(\GtlT Σ)_{W_H}) $ to the copy of itself in $ \HH(\GtlT Σ)_{W_H} $. From there, the inclusion $ (i_{W_H})^1 $ sends $ ν $ to its presentation as a cochain in $ \HC(\GtlT Σ) $. It can be traced directly with the help of the 0-adic and 1-adic components that these are the standard Hochschild cohomology representatives in $ \HC(\Gtl Σ) $ constructed in \cite{Paper-I, Paper-IV}.

For the second statement, recall that tensoring with a deformation base is a classical procedure \cite{Paper-IIA, Manetti}. Since $ i_{W_H} ∘ i' $ is a quasi-isomorphism of $ L_∞ $-algebras, so is its multilinear continuous extension after tensoring with the maximal ideal of the deformation base $ B $.

For the third statement, similar reasoning to \autoref{th:higher1-MCpush} in combination with the 0-adic and 1-adic components along the lines of \cite{Paper-I} yields the statement for Maurer-Cartan elements, and we finish the proof.
\end{proof}

\printbibliography

\bigskip
\begin{tabular}{@{}l@{}}%
    \textsc{Department of Mathematics, University of California, Berkeley, USA}\\
    \textit{jasper.kreeke@berkeley.edu}
  \end{tabular}

\end{document}